\documentclass[paper=a4, fontsize=11pt]{article} 

\usepackage[left=2.5cm, right=2.5cm, top=2.5cm, bottom=2.2cm]{geometry}

\usepackage[T1]{fontenc}
\usepackage[utf8]{inputenc}
\usepackage[french,english]{babel}

\usepackage{amsmath,amssymb,amsfonts,amsthm}
\usepackage{mathrsfs}

\usepackage{url}
\usepackage[colorlinks=true]{hyperref}
\hypersetup{colorlinks, 
            citecolor=[rgb]{0,0,0.7}, 
            linkcolor=[rgb]{0,0,0.7}
}

\usepackage{bbm}
\usepackage{enumitem}
\usepackage[numbers,comma,sort]{natbib}

\usepackage{xparse}
\usepackage{xargs}

\numberwithin{equation}{section}		

\newtheorem{theorem}{Theorem}[section]
\newtheorem{lemma}[theorem]{Lemma}
\newtheorem{corollary}[theorem]{Corollary}

\newtheorem{definition}[theorem]{Definition}
\newtheorem{proposition}[theorem]{Proposition}
\newtheorem{remark}[theorem]{Remark}

\newcommand{\R}{\mathbb{R}}

\newcommand{\N}{\mathbb{N}}
\newcommand{\EE}{\mathbb{E}}
\newcommand{\PP}{\mathbb{P}}

\newcommand{\dd}{~d}   

\newcommand{\pv}{p}
\newcommand{\ppv}{\frac{p}{2}}

\newcommand{\Cb}{\mathcal{C}_b}
\newcommand{\HH}{\mathcal{H}}
\newcommand{\RP}{\mathbf{X}}

\newcommand{\Vp}{\mathcal{V}^p}
\newcommand{\Vcalp}{\mathscr{V}^p}
\newcommand{\Cbeta}{\mathcal{C}^{\beta}}
\newcommand{\Ccalbeta}{\mathscr{C}^{\beta}}
\newcommand{\vertiii}[1]{{\left\vert\kern-0.25ex\left\vert\kern-0.25ex\left\vert #1 
		\right\vert\kern-0.25ex\right\vert\kern-0.25ex\right\vert}}

\newcommand{\Rsn}{R^{\sigma(Y^n)}}

\newcommand{\Ri}{R^{Y^i}}
\newcommand{\Rsi}{R^{\sigma(Y^i)}}

\newcommand{\kappab}{\makebox{\Large\ensuremath{\kappa}}}

\newcommandx{\flow}[4][1=\mathbf{X},2=y,3=t,4=0]{U_{#3\leftarrow #4}^{#1;#2}}
\newcommandx{\Jac}[4][1=\mathbf{X},2=y,3=t,4=0]{J_{#3\leftarrow #4}^{#1;#2}}

\begin{document}

\title{\normalfont \Large  {Penalisation techniques for one-dimensional reflected rough differential equations}}

\author{Alexandre Richard\thanks{Universit\'e Paris-Saclay, CentraleSup\'elec, MICS and CNRS FR-3487, France; \texttt{alexandre.richard@centralesupelec.fr}.}
\and Etienne Tanré \thanks{Universit\'e C\^ote d'Azur, Inria, France; \texttt{Etienne.Tanre@inria.fr}.}
\and Soledad Torres \thanks{Facultad de Ingenier\'ia, CIMFAV, Universidad de Valpara\'iso, Casilla 123-V, 4059 Valpara\'iso, Chile; \texttt{soledad.torres@uv.cl}. }
}

\date{\small \today}

\maketitle

\begin{abstract}
		In this paper we solve real-valued rough differential equations (RDEs) reflected on an irregular boundary. The solution $Y$ is constructed as the limit of a sequence $(Y^n)_{n\in\mathbb{N}}$ of solutions to RDEs with unbounded drifts $(\psi_n)_{n\in\mathbb{N}}$. The penalisation $\psi_n$ increases with $n$. Along the way, we thus also provide an existence theorem and a Doss-Sussmann representation for RDEs with a drift growing at most linearly. In addition, a speed of convergence of the sequence of penalised paths to the reflected solution is obtained. \\
	We finally use the penalisation method to prove that the law at time $t>0$ of some reflected Gaussian RDE is absolutely contiuous with respect to the Lebesgue measure.
\end{abstract}


\noindent {\sl MSC2010 Subject Classification\/}: 34F05, 60G15, 60H10.

\section{Introduction}

Solving (stochastic) differential equations with a reflecting boundary condition is by now a classical problem. For a domain $D\subseteq \R^e$, a mapping 
$\sigma:\R^e\rightarrow \R^{e\times d}$, 
an initial value $y_0\in \overline{D}$ and an $\R^d$-valued path $X=(X_t)_{t\in[0,T]}$ sometimes referred as the noise, this problem consists formally in finding $\R^e$-valued paths $(Y_t)_{t\in[0,T]}$ and $(K_t)_{t\in[0,T]}$ such that $\forall t\in[0,T]$,
\begin{align*}
&Y_t = y_0 + \int_0^t \sigma(Y_s) \dd X_s + K_t,\\
&Y_t\in \overline{D} ,~ |K|_T<\infty, \\
&|K|_t = \int_0^t \mathbf{1}_{\{X_s\in \partial D\}} \dd |K|_s~ \text{ and }~K_t = \int_0^t n(X_s) \dd |K|_s,
\end{align*}
where $|K|_t$ is the finite variation of $K$ on $[0,t]$ and $n(x)$ is the unit inward normal of $\partial D$ at $x$.
If $X$ is a Brownian motion and the integral is in the sense of Itô, this problem was first studied by \citet{Skorokhod}, and then by \citet{McKean}, \citet{ElKaroui}, \citet{LionsSznitman}, to name but a few. For this reason, it is called the Skorokhod problem associated to $X$, $\sigma$ and $D$ (see Definition \ref{def:SP}).

In the last few years, this problem has attracted a lot of attention when the driver $X$ is a $\beta$-Hölder continuous path: in the ``regular'' case $\beta\in(\tfrac{1}{2},1)$, existence of a solution has been established in a multidimensional setting by \citet{FerranteRovira} and uniqueness was then obtained by \citet{Slominski}. In that case, the integral can be constructed by a Riemann sum approximation and is known as a Young integral \cite{Young}. Extensions of these results to the ``irregular'' case $\beta<\tfrac{1}{2}$ can be handled with rough paths. We recall that this theory was initiated by \citet{Lyons98} and for a (multidimensional) $\beta$-Hölder continuous path $X$ and $\sigma$ a bounded vector field, it provides a way to solve the equation $\dd Y_t = \sigma(Y_t)\dd\RP_t$, where $\RP=(X,\mathbb{X})$ is the path $X$ with a supplementary two-parameters path $\mathbb{X}$ (in fact higher order correction terms such as $\mathbb{X}$ are needed if $\beta\leq \tfrac{1}{3}$, but we shall assume $\beta> \tfrac{1}{3}$ for simplicity). Solutions will be understood here as an equality between $Y_t$ and $y_0 + \int_0^t\sigma(Y_s)\dd\RP_s$ when this integral is defined in the sense of controlled rough paths \cite{Gubinelli,FrizHairer} (alternative approaches include the original definition of \citet{Lyons98}, the one of \citet{Davie} and  \citet{FrizVictoir}). Existence of solutions of reflected RDEs with $\beta\in(\tfrac{1}{3},\tfrac{1}{2})$ was proven by \citet{Aida} and \citet*{CastaingEtAl} under slightly different conditions. While \citet*{DeyaEtAl} proved uniqueness for a one-dimensional path reflected on the horizontal line. In those works, the existence is obtained through Wong-Zakai or Euler-type approximations, assuming that the boundary is either a convex or sufficiently smooth set, or a hyperplane. On the other hand in the Brownian setting, the reflected solutions have often been constructed by a penalisation procedure (see in particular \cite{LionsSznitman,ElKarouiEtAl,Slominski13}). The present work extends for the first time this classical technique to rough paths, which allows to cover the case of irregular boundaries.

We focus on one-dimensional ($e=1$) solutions to rough differential equations which are reflected on a moving boundary $L:[0,T]\rightarrow \R$, where the driver is a $d$-dimensional rough path $\RP$ with Hölder regularity $\beta\in(\tfrac{1}{3},1)$ (note that by a slight abuse of notations, we may use $\RP$ for $X$ and the vocabulary of rough paths even in the smooth case). It is proven that this problem has a unique solution, extending the result of \cite{DeyaEtAl} to a moving boundary. The existence is obtained with the following sequence of penalised RDEs:
\begin{equation}\label{eq:penalized_fSDE}
Y^n_t = y_0 + n\int_0^t (Y^n_s - L_s)_- \dd s + \int_0^t \sigma(Y^n_s) \dd\RP_s  .
\end{equation}
For technical reasons, the drift function $n(\cdot)_-$ will be replaced by a smooth function $\psi_n$ with at most linear growth, the interpretation remaining that of a stronger and stronger force pushing $Y^n$ above $L$. But unlike classical ODEs and SDEs, solving RDEs with unbounded coefficients is known to be tricky \cite{LejayRDE,LejayUnboundd,BailleulCatellier}. However, in case only the drift is unbounded (smooth and at most linearly growing) and $\sigma$ is smooth and bounded, \citet{RiedelScheutzow} proved the existence of a semiflow of solutions. Inspired by a result of \citet{FrizOberhauser}, we propose an alternative approach. 
We prove in Proposition~\ref{prop:ExistenceUnbdd} that any RDE with a drift having a bounded derivative has a unique global solution, which has a Doss-Sussmann--like representation \cite{Doss,Sussmann}. This last property turns to be extremely useful as it allows to transport the monotonicity of $(\psi_n)_{n\in\N}$ ($\psi_n\leq \psi_{n+1}$) to the penalised solution, leading to $Y^n\leq Y^{n+1}$. We are then able to prove the uniform convergence of $Y^n$ and $K^n_\cdot := \int_0^\cdot \psi_n(Y^n_s-L_s)\dd s$ to $Y$ and $K$, which are then identified as the solution to the Skorokhod problem described above. This reads (recall we assumed $e=1$):
\begin{equation}\label{eq:reflected_fSDE}
Y_t = y_0  + \int_0^t \sigma(Y_s) \dd \RP_s + K_t \quad \text{ and } \quad Y_t\geq L_t~,~\forall t\in[0,T],
\end{equation}
and the non-decreasing path $K$ increases only when $Y$ hits $L$. 
In addition, we obtain a rate of convergence in the previous result: the uniform distance between $Y^n$ and $Y$ is at most of order $n^{-\beta}$.  Up to a logarithmic factor, this result extends the optimal rate obtained in the Brownian framework by \citet{Slominski13}. Interestingly, our proof provides a new application of the rough Gr\"onwall lemma of \cite{DeyaEtAl0}.\\
Besides, when $\RP$ is a Gaussian rough path, the convergence of the sequence of penalised processes also happens uniformly in $L^\gamma(\Omega),~\gamma\ge 1$, and a rate of convergence is obtained.

The penalisation approach is a natural technique to solve reflected (ordinary, stochastic or rough) differential equations, and it also has fruitful applications to the study of the probabilistic properties of the solution. As an example, we prove that if $\sigma$ is constant, if $L\equiv 0$, and if the noise is a fractional Brownian motion with Hurst parameter $H\in[\tfrac{1}{2},1)$, then at each time $t>0$ the law of the solution $Y_t$ is absolutely continuous with respect to the Lebesgue measure on $(0,\infty)$. We expect to carry further investigations in this direction to relax the assumption on $\sigma$ and to get properties of the density.

\paragraph{Organisation of the paper.} In Section \ref{sec:prelim}, a brief overview of rough paths definitions and techniques is presented, followed by a set of precise assumptions and the statement of our main results. The existence of solutions for RDEs with unbounded drift (Proposition \ref{prop:ExistenceUnbdd}) is presented at the beginning of Section \ref{sec:penalisation} (the proof can be found in Appendix \ref{subsec:app1}), followed by the existence of a solution to the penalised equation and some penalisation estimates. The proofs that lead to the convergence of the penalised sequence to the reflected solution (Theorems \ref{th:RfSDE} and \ref{th:RfSDE_Gaussian}) are contained in the rest of Section \ref{sec:penalisation}: first it is proven that $Y^n$ and $K^n$ converge uniformly (we show monotone convergence of $Y^n$ towards a continuous limit), then that $Y$ is controlled by $X$ in the rough paths sense, which permits to use rough paths continuity theorems to show that $Y$ and $K$ solve the Skorokhod problem. In Section \ref{sec:rate}, we prove that there exists at most one solution to the one-dimensional Skorokhod problem with moving boundary (Theorem \ref{th:uniqueness}) and then we prove Theorem~\ref{th:rate} (and its probabilistic version Theorem~\ref{th:rateGaussian}), which gives a rate of convergence of the sequence of penalised paths to the reflected solution. In Section \ref{sec:density}, after recalling a few facts concerning Malliavin calculus and fractional Brownian motion, we prove that the reflected process with constant diffusion coefficient and driven by fractional noise admits a density at each time $t>0$ (Theorem~\ref{th:density}). 
Eventually, the \emph{a priori} estimates that are used in Sections~\ref{sec:penalisation} and \ref{sec:rate} are stated and proven in Section~\ref{subsec:app2}.

\paragraph{Notations.} We denote by $C$ a constant that may vary from line to line. For $k\in \N$ and $T>0$, $\Cb^k([0,T];F)$ (or simply $\Cb^k$) denotes the space of bounded functions which are $k$ times continuously differentiable with bounded derivatives, with values in some linear space $F$. If $E$ and $F$ are two Banach spaces, $\mathcal{L}(E,F)$ denotes the space of continuous linear mappings from $E$ to $F$. In the special case $E = \R^d$ and $F=\R$, we also write $(\R^d)'$ to denote the space of linear forms on $\R^d$. By a slight abuse of notations, we may consider row vectors as linear forms and vice versa. In this case, if $x\in \R^d$, the notation $x^T$ will be used for the transpose operation. The tensor product of two finite-dimensional vector spaces $E$ and $F$ is denoted by $E\otimes F$. In particular, $\R^d\otimes\R^e \simeq \R^{d\times e} \simeq \mathcal{M}^{d,e}(\R)$ is the space of real matrices of size $d\times e$.
\\
Let \(f\) be a function of one variable, and define
\begin{equation}
\delta f_{s,t} := f_t - f_s.
\end{equation}
The $2$-parameter functions are indexed by the simplex $\mathcal{S}_{[0,T]} = \{(s,t)\in[0,T]^2:~ s\leq t\}$ rather than $[0,T]^2$. If $I$ is a sub-interval of $[0,T]$, then $S_I = \{(s,t)\in I^2:~ s\leq t\}$. For $\beta\in(0,1)$ and a function $g:\mathcal{S}_{[0,T]} \rightarrow F$, the Hölder semi-norm of $g$ on a sub-interval $I\subseteq [0,T]$, denoted by $\|g\|_{\beta,I}$ (or simply $\|g\|_{\beta}$ if $I= [0,T]$), is given by
\begin{equation*}
\|g\|_{\beta,I} = \sup_{\substack{(s,t)\in \mathcal{S}_I\\s< t}} \frac{|g_{s,t}|}{(t-s)^\beta} .
\end{equation*}
The $\beta$-Hölder space $\Cbeta_2([0,T];F)$ is the space of functions $g:\mathcal{S}_{[0,T]}\rightarrow F$ such that $\|g\|_\beta<\infty$. The $\beta$-Hölder space $\Cbeta([0,T];F)$ is the space of functions $f:[0,T]\rightarrow F$ such that $\|\delta f\|_\beta<\infty$ (hereafter $\|\delta f\|_\beta$ will simply be denoted by $\|f\|_\beta$). With a slight abuse of notations, we may write $g\in \Cbeta([0,T];F)$ even for a $2$-parameter function, and if the context is clear, we may just write $g\in \Cbeta$. 

Similarly, we also remind the definitions of the \(p\)-variation semi-norm and space. For $p\geq 1$, a sub-interval $I\subseteq [0,T]$ and $g:\mathcal{S}_{[0,T]}\rightarrow F$, denote by $\|g\|_{p,I}$ (or simply $\|g\|_{p}$ if $I= [0,T]$) the semi-norm defined by
\begin{align*}
\|g\|_{p,I}^p = \sup_{\pi} \sum_{i=0}^{m-1} |g_{t_i,t_{i+1}}|^p ,
\end{align*}
where the supremum is taken over all finite  subdivisions $\pi = (t_0,\dots,t_m)$ of $I$ with $t_0 < t_1 < \cdots <t_m \in I$, $\forall m\in\N$. We define $\Vp_2$ the set of \emph{continuous} $2$-parameter paths \(g\) with finite \(p\)-variation, and $\Vp$ the set of \emph{continuous} paths $f:[0,T]\rightarrow F$ such that $\|\delta f\|_p <\infty$ (with the same abuse of notations, $\|\delta f\|_p$ will simply be denoted by $\|f\|_p$). 

Note that we shall use roman letters ($p$, $q$,...) for the variation semi-norms and greek letters ($\alpha$, $\beta$,...) for Hölder semi-norms in order not to confuse $\|\cdot\|_p$ and $\|\cdot\|_\alpha$. When $p=1$, we write $\|f\|_{1\text{-var}}$ to avoid confusion.

\begin{remark}
	The space \(\Cbeta\) (resp. \(\Vp\)) is Banach  when equipped with the norm 
	\(f \mapsto |f_0| + \|f\|_\beta.\) (resp. \(|f_0| + \|f\|_p\) ). When this property will be needed, 
	the paths will start from the same initial conditions, thus we may forget about the first term and consider $\|\cdot\|_\beta$ (resp. $\|\cdot\|_p$) as a norm.	
\end{remark}

Lastly, the mapping $\phi_p(x) = x\vee x^p,~x\geq 0$ will frequently appear in upper bounds of control functions that are used to control the $p$-variations of penalised and reflected solutions.

\section{Presentation of the Skorokhod problem and main results}\label{sec:prelim}

\subsection{Notations and definitions on rough paths and controlled rough paths}

In this section, we briefly review the definitions and notations of rough paths and rough differential equations, gathered mostly from \citet{FrizVictoir} and \citet{FrizHairer}. We also make precise the meaning of the Skorokhod problem written in Equation \eqref{eq:reflected_fSDE}.

For $\beta\in(\tfrac{1}{3},\tfrac{1}{2})$ (resp. $p\in(2,3)$), we denote by $\RP = \left((X_t)_{t\in[0,T]},(\mathbb{X}_{s,t})_{(s, t)\in\mathcal{S}_{[0,T]}}\right) \in \Cbeta([0,T];\R^d) \times \mathcal{C}^{2\beta}([0,T];\R^d\otimes \R^d)$ (resp. in $\Vp([0,T];\R^d) \times \mathcal{V}^{\frac{p}{2}}([0,T];\R^d\otimes \R^d)$) a rough paths, and denote by $\Ccalbeta_g([0,T];\R^d)$, or simply $\Ccalbeta_g$ (resp. $\Vcalp_g([0,T];\R^d)$ and $\Vcalp_g$) the space of geometric $\beta$-Hölder rough paths (resp. $p$-rough paths) with the following homogeneous rough path ``norm''
	\begin{align*}
	\vertiii{\RP}_{\beta} = \|X\|_\beta + \sqrt{\|\mathbb{X}\|_{2\beta}} ~~~(\text{resp. } \vertiii{\RP}_{p}^p = \|X\|_p^p + \|\mathbb{X}\|_{\frac{p}{2}}^{\frac{p}{2}}).
	\end{align*}
Hereafter, we use the notation $\RP\in \Ccalbeta_g$ even if $\beta>\tfrac{1}{2}$, although the iterated integral $\mathbb{X}$ is irrelevant in this case. This notation permits to present our results in a unified form. Our main results are expressed in Hölder spaces, but the $p$-variations play an important role in the proofs, due to the nature of the compensator process $K$ (which is non-decreasing and thus in $\mathcal{V}^1$).

For a geometric rough path $\RP\in \Ccalbeta_g\left([0,T],\R^d\right)$, one would like to give a meaning to the following equation:
\begin{align}\label{eq:RDEdrift}
\dd Y_t = b(Y_t)\dd t + \sigma(Y_t) \dd \RP_t . 
\end{align}
We choose here to formulate the problem in the framework of controlled rough paths. We recall it here for the notations (see \cite[Definition 4.6]{FrizHairer}).

\begin{definition}[Controlled rough path]\label{def:controlledRP}
	Let  $X\in \Cbeta([0,T];\R^d)$ (resp. $X\in \Vp([0,T];\R^d)$). A path $Y\in \Cbeta([0,T];E)$ (resp. $Y\in \Vp([0,T];E)$) is \emph{controlled} by $X$ if there exist a path $Y'\in \Cbeta([0,T];\mathcal{L}(\R^d,E))$ (resp. $Y'\in \Vp([0,T];\mathcal{L}(\R^d,E))$) and a map $R^Y \in \mathcal{C}^{2\beta}_2([0,T];E)$ (resp. $R^Y \in \mathcal{V}^{\frac{p}{2}}_2([0,T];E)$) such that
	\begin{align*}
	\forall (s, t)\in \mathcal{S}_{[0,T]},\quad \delta Y_{s,t} = Y'_s \delta X_{s,t} + R^Y_{s,t}.
	\end{align*}
	The path $Y'$ is called the Gubinelli derivative of \(Y\) (although it might not be unique), and $R^Y$ is a remainder term. The space of such couples of paths $(Y,Y')$ controlled by $X$ is denoted by $\Cbeta_X(E)$ (resp. $\Vp_X(E)$).
\end{definition}

Now if $\RP$ is a geometric rough path, the rough integral of $Y$ against $\RP$ is classically defined by 
\begin{align}\label{def:RoughInt}
\int_0^T Y_s \dd\RP_s = \lim_{m \to \infty} \sum_{\pi_m = (t_i^m)} Y_{t_i^m} \delta X_{t_i^m,t_{i+1}^m} + Y'_{t_i^m} \mathbb{X}_{t_i^m,t_{i+1}^m} ,
\end{align}
where $(\pi_m)_{m\in{\N}}$ is an increasing sequence of subdivisions of $[0,T]$ such that  $\displaystyle\lim_{m\rightarrow \infty}\max_i(t_{i+1}^m-t_i^m) =0$ and $t_0^m=0,~t_m^m=T$. The existence of this integral has been established by \citet{Gubinelli} for the Hölder topology (see also \cite[Proposition 4.10]{FrizHairer}). For technical reasons related to the nature of the compensator term $K$, which is clearly in $\mathcal{V}^1$ but not so clearly H\"older continuous, it will be convenient (see Lemma~\ref{lem:RDeltaSig} and its proof) to use a similar result in $p$-variation topology.
\begin{theorem}[\citet{FrizShekhar}, Theorem 31]\label{th:RI}
	Let $p\in[2,3)$. If $\RP\in \Vcalp_g([0,T];\R^d)$ and $(Y,Y')\in \Vp_X(\mathcal{L}(\R^d,\R^e))$, then the rough integral of $Y$ against $\RP$ exists (and the limit in \eqref{def:RoughInt} does not depend on the choice of a sequence of subdivisions). Moreover, for any $(s, t)\in \mathcal{S}_{[0,T]}$,
	\begin{align}\label{eq:boundRI}
	\left|\int_s^t Y_u \dd \RP_u -Y_s \delta X_{s,t} - Y_s' \mathbb{X}_{s,t}\right| \leq C_p \left( \|X\|_{\pv,[s,t]} \|R^{Y}\|_{\ppv,[s,t]} + \|\mathbb{X}\|_{\ppv,[s,t]} \|Y'\|_{\pv,[s,t]}\right).
	\end{align}
\end{theorem}

Lastly, recall the following definition, borrowed from \cite[Def. 1.6]{FrizVictoir}.
\begin{definition}[Control function]
	Let $I$ be an interval and recall that $\mathcal{S}_I$ denotes the simplex on $I$. 
	A control function is a map $w:\mathcal{S}_I\rightarrow \R_+$ which is 
	\begin{itemize}
	\item[-] super-additive, 
	i.e. $w(s,t)+w(t,u)\leq w(s,u)$ for all $s\leq t\leq u \in I$, 
	\item[-] continuous on \(\mathcal{S}_I\),
	\item[-] zero on the diagonal, i.e. $ w(s,s)=0$.
	\end{itemize}
\end{definition}
\noindent For instance, if $X\in\Vp(I)$ for some interval $I\subseteq [0,T]$, then $w_X(s,t) = \|X\|_{p,[s,t]}^p$ is a control function on $\mathcal{S}_I$.

\subsection{The Skorokhod problem}

Having at our disposal a rough integral in the sense of Equation \eqref{def:RoughInt}, we can give a meaning to Equation \eqref{eq:reflected_fSDE}, also referred to as Skorokhod problem associated to $\sigma$ and $L$, denoted by $SP(\sigma,\RP,L)$.
\begin{definition}\label{def:SP}
	Let $\RP\in \Ccalbeta_g([0,T];\R^d)$. We say that $(Y,K)$ solves $SP(\sigma,\RP,L)$, or that it is a solution to the reflected RDE with diffusion coefficient $\sigma$ started from $y_0\geq L_0$ and reflected on the path $L$, if 
	\begin{enumerate}[label = (\roman*)]
		\item\label{item1} $(Y,\sigma(Y))\in \Vp_X$ and $(Y,K)$ satisfies Equation \eqref{eq:reflected_fSDE}, in the sense that both sides are equal, where the integral $\int_0^\cdot \sigma(Y_s) \dd\RP_s$ is understood in the sense of \eqref{def:RoughInt}; \label{def:SP1}
		\item $\forall t\in[0,T],~Y_t\geq L_t$; \label{def:SP2}
		\item $K$ is non-decreasing; \label{def:SP3}
		\item $\forall t\in[0,T]$, $\int_0^t (Y_s - L_s)  \dd K_s = 0$, or equivalently, $\int_0^t \mathbf{1}_{\{Y_s\neq L_s\}} \dd K_s = 0$ . \label{def:SP4}
	\end{enumerate}
\end{definition}

\begin{remark}
	In item \ref{item1}, it is also possible to define solutions to reflected RDEs in the sense of Davie as in \citet{Aida} and \citet{DeyaEtAl}. 
\end{remark}

\subsection{Main results}

As an extension to a moving boundary of the theorem of \cite{DeyaEtAl}, we have the following result:
\begin{theorem}\label{th:uniqueness}
Let $\RP = (X,\mathbb{X})\in \Ccalbeta_g([0,T];\R^d)$ for some $\beta\in(\tfrac{1}{3},1)\setminus\{\tfrac{1}{2}\}$, let $L\in \mathcal{C}^\alpha([0,T];\R)$ with $\alpha>\tfrac{1}{2}\vee (1-\beta)$, and assume that $\sigma \in \mathcal{C}_b^3(\R,(\R^d)')$. Then there is at most one solution to the problem $SP(\sigma,\RP,L)$ with initial condition $y_0\geq L_0$.
\end{theorem}

In particular, this theorem is crucial to ensure that the solution we construct does not depend on a choice of penalisation.  

Since the penalisation terms $n (\cdot)_-$ in \eqref{eq:penalized_fSDE} are not differentiable, we replace them by  smooth functions $\psi_n$  such that
\begin{equation}\label{eq:hyp_psin}
\forall n\in\N,\quad 
\begin{cases}
\psi_n\in\mathcal{C}^\infty, ~\psi_n' \in \Cb^\infty ~\text{ and } ~ \psi_n'\leq 0;\\
\forall y\in\R,\quad (-\tfrac{1}{2} + ny_-)\vee 0 \leq \psi_n(y) \leq ny_-;\\
\psi_n \leq \psi_{n+1}.
\end{cases}
\end{equation}
With these notations and assumptions, one can consider the penalised paths defined by
\begin{equation}\label{eq:penalized_smoothed_fSDE}
Y^n_t = y_0  + \int_0^t \psi_n(Y^n_s - L_s) \dd s + \int_0^t \sigma(Y^n_s) \dd \RP_s , \quad t\in[0,T].
\end{equation}
For each $n\in\N$, Proposition~\ref{prop:globalExistence} will ensure that there is a unique solution to \eqref{eq:penalized_smoothed_fSDE},  understood in the sense that the left-hand side of \eqref{eq:penalized_smoothed_fSDE} equals its right-hand side with the last integral defined as in \eqref{def:RoughInt}. This is new since $\psi_n$ is unbounded.

\begin{theorem}\label{th:RfSDE}
	Let $\RP = (X,\mathbb{X})\in \Ccalbeta_g([0,T];\R^d)$ for some $\beta\in(\tfrac{1}{3},1)\setminus\{\tfrac{1}{2}\}$ and 
	let $L\in \mathcal{C}^\alpha([0,T];\R)$ with $\alpha>\tfrac{1}{2}\vee (1-\beta)$. Assume that $\sigma \in \mathcal{C}_b^4(\R,(\R^d)')$, 
	that $(\psi_n)_{n\in\N}$ satisfies \eqref{eq:hyp_psin}, 
	and that $y_0\geq L_0$. For each $n\in \N$, $Y^n$ denotes the solution to \eqref{eq:penalized_smoothed_fSDE}.
	
Then $\left((Y^n_t , \int_0^t \psi_n(Y^n_s-L_s)\dd s)_{t\in[0,T]}\right)_{n\in\N}$ converges uniformly on $[0,T]$ to $(Y,K)$, the solution to the Skorokhod problem $SP(\sigma, \RP, L)$, and $Y \in \Cbeta$.
\end{theorem}

~

Furthermore, we obtain a rate of convergence of the sequence of penalised processes to the reflected solution.
\begin{theorem}\label{th:rate}
	Assume that the hypotheses of Theorem~\ref{th:RfSDE} hold. 
 Then the penalised solution $Y^n$ converges to $Y$ with the following rate, for some $C>0$:
	\begin{align*}
	\forall n\in\N^*,~\forall t\in[0,T],\quad 0\leq Y_t - Y^n_t \leq C~n^{-\beta} .
	\end{align*}
\end{theorem}
Compared with the Theorem 4.1 of \citet{Slominski13} for reflected diffusions, we see that our result matches the optimal rate, up to a logarithmic correction. However the result of \citet{Slominski13} is in $L^p(\Omega)$ whereas the previous theorem is pathwise. We will be able to close this gap partially in Theorem~\ref{th:rateGaussian} when $\RP$ is a Gaussian rough path.

\subsection{Main results for Gaussian-driven RDEs}

In case $X$ is a Gaussian process, several papers give conditions (see in particular \citet*{CassHairerLittererTindel}) for $X$ to be enhanced into a geometric rough path. \citet*{CassLittererLyons} also proved that such conditions yield that the Jacobian of the flow has finite moments of all order (see also \cite{CassHairerLittererTindel} with a bounded drift).

Let $(\Omega,\mathcal{F},\PP)$ be a complete probability space, and let $X = (X^1,\dots X^d)$ be a continuous, centred Gaussian process with independent and identically distributed components and let $R(s,t) = \EE\left(X^1_s X^1_t\right)$ denote the covariance function of $X^1$. Following \citet{CassLittererLyons}, let
\begin{align*}
R\left(\begin{array}{c}
s,t\\
u,v
\end{array}  \right) = \EE\left[(X^1_t-X^1_s)(X^1_v-X^1_u)\right]
\end{align*}
be the rectangular increments of $R$. Then for $r\in[1,\tfrac{3}{2})$, we might assume that $R$ has finite second-order $r$-variation in the sense
\begin{align}\label{hyp:XGauss}
\left\| R \right\|_{r;[0,T]^2} := 
\left(\sup_{\substack{\pi = (t_i)\\\pi' = (t_j')}} \sum_{i,j} R
\left(
\begin{array}{c}
t_i,t_{i+1}\\
t_j',t_{j+1}'
\end{array}  \right)^r
\right)^{\frac{1}{r}}<\infty . \tag{$\text{H}_{\text{Cov}}$}
\end{align}
Under this assumption, $X$ can almost surely be enhanced into a geometric rough path $\RP = (X,\mathbb{X})$ and for any $\alpha \in (\tfrac{1}{3},\tfrac{1}{2r})$, $\RP \in \mathscr{C}^{\alpha}_g$.
Moreover, this assumption permits to obtain upper bounds on the Jacobian of the 
flow of a Gaussian RDE, which shall help us to obtain convergence results in $L^\gamma(\Omega)$ (Theorem~\ref{th:RfSDE_Gaussian}).

\begin{remark}\label{rk:def_fBm}
	A typical example of process satisfying \eqref{hyp:XGauss} is the fractional Brownian motion $(B^H_t)_{t\geq 0}$. We recall that for any Hurst parameter $H\in(0,1)$, $(B^H_t)_{t\geq 0}$ is the centred Gaussian process with covariance
	\begin{align*}
	\EE\left(B^H_t B^H_s \right) = \frac{1}{2} \left(t^{2H} + s^{2H} - |t-s|^{2H} \right) ,\quad \forall t,s\geq 0 .
	\end{align*}
	Such a process is statistically $H$--self-similar and increment stationary (e.g. for $H=\tfrac{1}{2}$, this is a standard Brownian motion). Most importantly regarding the theory of rough paths, if $H\in(\tfrac{1}{3},\tfrac{1}{2}]$, its covariance satisfies \eqref{hyp:XGauss} with $r=\tfrac{1}{2H}$, so that it can be enhanced into a geometric rough paths. If $H\in(\tfrac{1}{2},1)$, then one can solve differential equations driven by $B^H$ in the Young sense (i.e. without needing to enhance $B^H$). Besides, its sample paths are almost surely $\beta$-Hölder continuous, for any $\beta<H$.
\end{remark}

\begin{theorem}\label{th:RfSDE_Gaussian}
	Let $\sigma\in \mathcal{C}^4_b(\R,(\R^d)')$, let $(\psi_n)_{n\in\N}$ satisfy condition \eqref{eq:hyp_psin} and $y_0\geq L_0$ almost surely. Let $X = (X^1,\dots , X^d)$ be an a.s. continuous, centred Gaussian process with independent and identically distributed components, and let $R$ be its covariance function. Assume that either $X\in \Cbeta([0,T];\R^d)~a.s.$ and $L\in \Cbeta([0,T];\R)~a.s.$ for some $\beta\in (\tfrac{1}{2},1)$, or that:
	\begin{itemize}[parsep=0cm,itemsep=0cm,topsep=0.1cm]
		\item $R$ has finite second-order $r$-variations for some $r\in[1,\tfrac{3}{2})$, as in \eqref{hyp:XGauss};
		\item and $L\in \mathcal{C}^\alpha([0,T];\R)~a.s.$ for some $\alpha > 1-\tfrac{1}{2r}$ and that $\EE\left[\|L\|_{\alpha}^\gamma \right]<\infty$, for any $\gamma\geq 1$. 
	\end{itemize}
	Then the conclusions of Theorem~\ref{th:RfSDE} hold in the almost sure sense and moreover, the convergence holds in the following sense: for any $\gamma\geq 1$, 
	\begin{equation}\label{eq:convYnGauss}
	\lim_{n\rightarrow +\infty} \EE\left[\sup_{t\in[0,T]} |Y_t - Y^n_t|^\gamma\right] = 0 .
	\end{equation}
\end{theorem}

As a follow-up to the remark initiated after Theorem~\ref{th:rate}, one may wonder if a rate of convergence works in $L^\gamma(\Omega)$, so as to compare it with Theorem 4.1 of \citet{Slominski13}. But observe that Theorem~\ref{th:rate} is proven through a Gr\"onwall argument and the constant $C$ appearing there is thus of exponential form. Besides, the $p$-variation norm of $J^\RP$ (the Jacobian of the flow of the RDE) appears in this exponential, and only has sub-exponential moments \cite[Theorem 6.5]{CassLittererLyons}. This explains the logarithm appearing in the following result.
\begin{theorem}\label{th:rateGaussian}
	Assume that the hypotheses of Theorem~\ref{th:RfSDE_Gaussian} hold. 
 Then for any $\gamma\geq 1$,
	\begin{align*}
	\EE\left[  \log \left(1+\sup_{n\in\N} \big(n^\beta \|Y-Y^n\|_{\infty,[0,T]} \big)\right)^\gamma \right]  <\infty .
	\end{align*}
\end{theorem}

The last result of this paper is a nice application of the previous penalisation technique and results, which are used to prove the existence of a density for the reflected process when the noise is a fractional Brownian motion. It is presented under simplified assumptions as the general case would be out of the scope of the present paper and will be further investigated in a separate work.
\begin{theorem}\label{th:density}
	Let $B^H$ be a one-dimensional fractional Brownian motion with Hurst parameter $H\in[\tfrac{1}{2},1)$, and let $b\in \Cb^1$. Let $(Y,K)$ be the solution to the Skorokhod problem reflected on the horizontal axis
	\begin{equation*}
	\forall t\geq 0,\quad Y_t = y_0 + \int_0^t b(Y_s) \dd s + K_t + B^H_t ,\quad Y_t\geq 0.
	\end{equation*}
	Then for any $t>0$, the restriction of the law of $Y_t$ to $(0,\infty)$, i.e. the measure ${[\mathbf{1}_{\{Y_t>0\}} \PP]\circ Y_t^{-1}}$, admits a density with respect to the Lebesgue measure.
\end{theorem}

Note that unless otherwise stated (mostly in Section~\ref{sec:density}), we will only consider the case $\beta\in(\tfrac{1}{3},\tfrac{1}{2})$. Indeed if $\beta\in(\tfrac{1}{2},1)$, Young integrals can be used, which makes proofs easier.

\section{Existence of reflected RDEs by penalisation}\label{sec:penalisation}

In this section, we will prove Theorems~\ref{th:RfSDE} and \ref{th:RfSDE_Gaussian}. The outline of the proof is as follows: 

First, in Subsection~\ref{subsec:flow}, the solution of the penalised equation $Y^n$ is expressed in Doss-Sussman form, involving the rough path flow of some auxiliary process $Z^n$. This is then used to prove that the sequence $(Y^n)_{n\in\N}$ is non-decreasing. \\
In Subsection~\ref{subsec:existence}, a general result about penalised differential equations (Lemma~\ref{lem:deterministic}) allows to obtain, by comparison of ODEs, an upper bound on $Y^n$ which is independent of $n$ (Lemma~\ref{lem:comparison}). This leads to the proof of existence of $Y$ as a pointwise non-decreasing limit (Proposition~\ref{lem:convYX}).\\
To identify the limit $Y$ as the solution to $SP(\sigma,\RP,L)$, we first derive \emph{a priori} estimates on the $p$-variation of $\sigma(Y^n)$ and $R^{\sigma(Y^n)}$ (Lemmas~\ref{lem:unifBoundRYn} and \ref{lem:RDeltaSig}). The reason the $p$-variation appears here instead of the Hölder norm is explained in Subsection~\ref{subsec:app2}.\\
In Subsection~\ref{subsec:unifcont}, we use these \emph{a priori} estimates to deduce that $\int_0^\cdot \sigma(Y^n_u) \dd\RP_u$ is uniformly (in $n$) Hölder continuous (Corollary~\ref{cor:unifHolder}). By using again Lemma~\ref{lem:deterministic}, it follows that $\|(Y^n-L)_-\|_\infty \rightarrow 0$, and since this term appears in the upper bound on $Y^n_t-Y^n_s$ (in Lemma~\ref{lem:comparison}), one now obtains that $Y$ is continuous. By Dini's theorem, this implies that $Y^n$ converges uniformly to $Y$ (Proposition~\ref{prop:unifConvYn}). \\
In Subsection~\ref{subsec:identification}, the uniform convergence and the \emph{a priori} estimates of Lemma~\ref{lem:unifBoundRYn} are used to prove that $\int_0^\cdot \sigma(Y^n_u) \dd\RP_u$ converges uniformly to $\int_0^\cdot \sigma(Y_u) \dd\RP_u$. Hence the penalisation term $K^n:=\int_0^\cdot \psi_n(Y^n_u-L_u) \dd u$ also converges uniformly to some path $K$ such that $(Y,K)$ is the solution of $SP(\sigma,\RP,L)$.

\subsection{Flow of an RDE and existence of a global solution to \eqref{eq:penalized_smoothed_fSDE}}\label{subsec:flow}

We aim at solving \eqref{eq:RDEdrift} with an unbounded drift $b$. Recall that the Doss-Sussmann representation (\citet{Doss} and \citet{Sussmann}) provides the solution of an SDE with one-dimensional noise as the composition of the flow of $\sigma$ with the solution of a random ODE. However its multidimensional generalization requires strong geometric assumptions on $\sigma$ (see \cite{Doss}). Instead, \citet{FrizOberhauser} established a less explicit but similar formulation in terms of flows, requiring that both $\sigma$ and $b$ are bounded with bounded derivatives.

Note that when the vector fields are unbounded (which is the case of $\psi_n$), known counter-examples show that global solutions of RDEs may not exist. Nevertheless, for an RDE with coefficient $V = (V_1,\dots,V_d)$ on $\R^e$, there are several existence results (\cite{LejayRDE,LejayUnboundd}, \cite[Exercise 10.56]{FrizVictoir} and \cite{BailleulCatellier}) but they are not applicable here.
However this general approach neglects the special nature of the drift term and its smooth driver ``$dt$'' by considering it as any other component of the rough driver. On the other hand, \citet{RiedelScheutzow} proved that under a linear growth assumption on $b$,  \eqref{eq:RDEdrift} has a unique semiflow of solutions. Under similar assumptions, we provide here a Doss-Sussmann representation of the solution.

Let $\beta\in (\tfrac{1}{3},\tfrac{1}{2})$, $\RP\in \Ccalbeta_g\left([0,T];\R^d\right)$ and $\sigma\in \Cb^4(\R^e;\mathcal{L}(\R^d,\R^e))$. Consider the RDE
\begin{align}\label{eq:RDE}
\dd \tilde{Y}_t = \sigma(\tilde{Y}_t)\dd \RP_t ,
\end{align}
and  denote by $\left(\flow[\RP][y_0]\right)_{t\in[0,T]}$ the solution started from $y_0$, $y_0\mapsto \flow[\RP][y_0]$ the flow of the solution, $\Jac[\RP][y_0]$ its Jacobian and $\Jac[\RP][y_0][0][t]$ the inverse of the Jacobian (note that $\sigma\in \Cb^3$ is enough for existence and uniqueness in \eqref{eq:RDE}). We know that the smoothness of the flow depends on the smoothness of $\sigma$: for any $t$, $y_0\mapsto\flow[\RP][y_0]$ is Lipschitz continuous and twice differentiable (see for instance \cite[Proposition 3]{FrizOberhauser}). According to \cite[Corollary 4.6]{CassLittererLyons}, $\Jac[\RP][y_0]$ is uniformly (in $(t,y_0)\in [0,T]\times \R$) bounded by a quantity depending only on $p=\beta^{-1}$, $\|\RP\|_{\pv,[0,T]}$ and the so-called $\alpha$-local $p$-variation of $\RP$ (see \cite[Definition 4.3]{CassLittererLyons}). We denote this upper bound by $C^{\RP}_J$.\\
The inverse of the Jacobian $\Jac[\RP][\cdot][0][t] = (\Jac[\RP][\cdot])^{-1}$ is also the Jacobian of the flow of the same RDE with $X$ evolving backward. 
Thus, as noticed in the proof of \cite[Theorem 7.2]{CassHairerLittererTindel}, $\Jac[\RP][\cdot][0][t]$ is also bounded by $C^{\RP}_J$:
\begin{align}\label{eq:boundUJJinv}
\sup_{y_0\in\R^e} \max\left(\|\Jac[\RP][y_0][\cdot]\|_{\infty,[0,T]}, \|\Jac[\RP][y_0][0][\cdot]\|_{\infty,[0,T]} \right) \leq C^{\RP}_J <\infty .
\end{align}
Note also that with $\sigma\in\Cb^4$, $\Jac[\RP][\cdot][0][t]$ and $\Jac[\RP][\cdot]$ are Lipschitz continuous, uniformly in $t$.\\
Besides, when $X$ is Gaussian with \textit{i.i.d.} components and satisfies \eqref{hyp:XGauss}, $C^\RP_J$ has finite moments of all orders (\cite[Theorem 6.5]{CassLittererLyons} and \cite[Theorem 7.2]{CassHairerLittererTindel}).
As observed in \cite[Section 7]{CassHairerLittererTindel}, $\Jac[\RP][z][0][\cdot]$ satisfies the following linear RDE, for any fixed $z$:
\begin{align*}
d\Jac[\RP][z][0][t] =d\mathbf{M}_t \Jac[\RP][z][0][t] ,
\end{align*}
where $\mathbf{M}$ depends on the flow $\flow[\RP][z]$. If $e=1$ (recall $e$ is the dimension of the space in which $y$ lives), it is thus a consequence of the fact that $\Jac[\RP][z][0] = 1$ and of the uniqueness in the previous equation that $\Jac[\RP][z][0][t]>0$ for any $z\in\R$ and any $t\geq 0$. Hence it follows from \eqref{eq:boundUJJinv} that
\begin{align}\label{eq:lowBoundJinv}
\forall z\in \R,\quad \Jac[\RP][z][0][t] \geq (C^{\RP}_J)^{-1} (>0).
\end{align}

\begin{proposition}\label{prop:ExistenceUnbdd}
	Let $d,e\in\N^*$. Let $\sigma\in\Cb^4(\R^e, \mathcal{L}(\R^d;\R^e))$ and assume that
	\begin{equation*}
	\quad b\in \mathcal{C}^1(\R^e,\R^e) ~\text{ and }~\nabla b\in \Cb(\R^e,\R^{e\times e}) .
	\end{equation*}
	Let $\beta\in(\tfrac{1}{3},\tfrac{1}{2})$, and let $\RP\in \Ccalbeta_g\left([0,T],\R^d\right)$. Then for any initial condition $y_0$, there exists a unique $(Y,Y') \in \mathcal{C}_X^\beta$ such that $Y$ solves Eq. \eqref{eq:RDEdrift} on $[0,T]$ in the sense of controlled rough paths, and $Y' = \sigma(Y)$. Moreover,
	\begin{equation*}
	\left\{
	\begin{array}{ll}
	Y_t &=\flow[\RP][Z_t] \\
	Z_t &=  y_0 + \int_0^t  W(s,Z_s) \dd s 
	\end{array}
	\right.
	,~t\in[0,T],
	\end{equation*}
	where 
	\begin{equation*}
	W(t,z) = \Jac[\RP][z][0][t]~ b\left(\flow[\RP][z]\right), ~(t,z)\in [0,T]\times \R^e.
	\end{equation*}
\end{proposition}
The proof is postponed to the Appendix~\ref{subsec:app1}. The idea is to derive first the local existence and a Doss-Sussmann representation on a small time interval. Global existence is then achieved by stability of the ODE in the Doss-Sussmann representation. This extends a result of \citet{FrizOberhauser} to unbounded drifts.

As a consequence of Proposition~\ref{prop:ExistenceUnbdd} and up to a slight adaptation due to the boundary term $L$, there is a global solution to the penalised RDE \eqref{eq:penalized_smoothed_fSDE}.

\begin{proposition}\label{prop:globalExistence}
	Let $\sigma\in\Cb^4(\R;(\R^d)')$, $n\in\N$, and $\psi_n$ satisfying \eqref{eq:hyp_psin}. Let $\beta\in(\tfrac{1}{3},\tfrac{1}{2})$, let $\RP\in \Ccalbeta_g\left([0,T],\R^d\right)$ and let $(L_t)_{t\in[0,T]}$ be a barrier in $\mathcal{C}^\alpha([0,T],\R)$, with $\alpha >1-\beta$. Then for any initial condition $y_0$ such that $y_0\geq L_0$, there exists a unique $(Y^n,{Y^n}') \in \mathcal{C}_X^\beta$ such that $Y^n$ solves Eq. \eqref{eq:penalized_smoothed_fSDE} on $[0,T]$ in the sense of controlled rough paths, and ${Y^n}' = \sigma(Y^n)$. Moreover,
	\begin{equation}\label{eq:def_XY}
	\left\{
	\begin{array}{ll}
	Y_t^n &=\flow[\RP][Z^n_t] \\
	Z_t^n &=  y_0 + \int_0^t  W^n(s,Z^n_s) \dd s 
	\end{array}
	\right.
	,~t\in[0,T],
	\end{equation}
	where 
	\begin{equation*}
	W^n(t,z) = \Jac[\RP][z][0][t]~ \psi_n\left(\flow[\RP][z] -L_t\right), ~(t,z)\in [0,T]\times \R.
	\end{equation*}
\end{proposition}

\begin{remark}\label{rem:sYn}
	\begin{itemize}[parsep=0cm,itemsep=0cm,topsep=0.1cm]
		\item The path $\sigma(Y^n)$ is controlled by $X$ and its Gubinelli derivative is $\sigma'(Y^n) \sigma(Y^n)$ ($\sigma'(y)$ is an element of $\mathcal{L}(\R,\mathcal{L}(\R^d,\R))$). In particular,  $\Rsn\in \mathcal{C}^{2\beta}([0,T],(\R^d)')$ (see \cite[Lemma 7.3]{FrizHairer}).
		\item For $\beta>\tfrac{1}{2}$, our assumptions on the coefficients meet those from \cite{HuNualart} and thus there exists a unique solution to (\ref{eq:penalized_smoothed_fSDE}). Moreover, the previous Doss-Sussmann representation holds also true by a simple application of the usual chain rule for Young integrals.
	\end{itemize}
\end{remark}

\begin{proof}
	For $y\in \R^2$, define $\widehat{b}_n(y) = (\psi_n(y^1-y^2),0)^T$, where we used the notation $y=(y^1,y^2)^T\in \R^2$. In the same way, define $\widehat{\sigma}(y) = \left(\begin{array}{cc} \sigma(y^1) & 0\\ 0 & 1\end{array}\right)$, so that $\widehat{\sigma}\in \Cb^4\left(\R^2;\mathcal{L}(\R^{d+1},\R^2)\right)$. Finally, since $\alpha+\beta>1$, let $\widehat{\RP} \in \Ccalbeta_g$ be the Young pairing of $\RP$ and $L$ (see \cite[Section 9.4]{FrizVictoir}). Proposition~\ref{prop:ExistenceUnbdd} ensures that there exists a unique solution $\widehat{Y}^n \in \Cbeta([0,T];\R^2)$ to the following RDE with drift
	\begin{align*}
	\dd\widehat{Y}^n_t = \widehat{b}_n(\widehat{Y}^n_t)\dd t + \widehat{\sigma}(\widehat{Y}^n_t)\dd\widehat{\RP}_t .
	\end{align*}
	Since $Y^n$ corresponds to the first component of $\widehat{Y}^n$ and $(Y^n,L)$ is controlled by $\widehat{X}$, it is not difficult to check that $Y^n$ is controlled by $X$ (with Gubinelli derivative $\sigma(Y^{n})$), and the result follows.
\end{proof}

\subsection{Existence of the limit path}\label{subsec:existence}

The following result will be used several times in the sequel. It gives uniform estimates for solutions of integral equations with drift coefficient $\psi_n$. The proof is postponed to the Appendix.

\begin{lemma}\label{lem:deterministic}
	Let $\Psi>0$, $\ell , (g^n)_{n\in\N}$ be continuous functions such that $g^n_0=0$, and assume that for each $n\in\N$, $f^n$ is a solution to:
	\begin{equation*}
	\begin{cases}
	&f^n_t = f^n_0 + g^n_t + \Psi\int_0^t \psi_n(f^n_u-\ell_u) \dd u,\quad \forall t\in[0,T], \\
	&f^n_0 = f_0 \geq \ell_0.
	\end{cases}
	\end{equation*}
	\begin{enumerate}[label = (\roman*)]
		\item Then for all $t\in[0,T]$,
		\begin{align*}
		\forall n\in\N,~ |\delta f^n_{0,t}-\delta\ell_{0,t}| \leq \sqrt{26} \|g^n_\cdot -\delta\ell_{0,\cdot}\|_{\infty,[0,t]} ~;
		\end{align*}

		\item Let $\beta\in(0,1)$. If $\ell , (g^n)_{n\in\N}\in \Cbeta([0,T],\R)$ and $f^n_0\geq \ell_0$, then
		\begin{equation*}
		\forall t\in [0,T],~ \forall n\in\N, \quad  \psi_n(f^n_t-\ell_t) \leq \overline{\Psi}_n (\Psi^{-\beta} +\Psi^{1-\beta}) n^{1-\beta},
		\end{equation*}
		where $\overline{\Psi}_n = C (\|\ell\|_{\beta} +  \|g^n\|_{\beta} + \tfrac{1}{2}\Psi T^{1-\beta})$.
	\end{enumerate}
\end{lemma}

We now propose a comparison result. Jointly with the Doss-Sussmann representation \eqref{eq:def_XY} and Lemma~\ref{lem:deterministic}, it will imply the existence of $Z$ and $Y$ as pointwise limits of $(Z^n)$ and $(Y^n)$ (Proposition~\ref{lem:convYX}).

\begin{lemma}\label{lem:comparison}
Let $n\in\N$ and $Y^n$ be the solution of \eqref{eq:penalized_smoothed_fSDE} given by Proposition~\ref{prop:globalExistence}. Recall that $Z^n$ is the solution of the random ODE in \eqref{eq:def_XY}.
\begin{enumerate}[label = (\roman*)]
\item\label{item:comparisoni} There exists $C>0$ which depends only on $\sigma,\beta,T$ such that for any $(s,t)\in \mathcal{S}_{[0,T]}$,
\begin{align*}
	Z^n_s\leq Z^n_t \leq Z^n_s + C~C^\RP_J~\left((\vertiii{\RP}_{\beta,[0,T]} \vee \vertiii{\RP}_{\beta,[0,T]}^{\frac{1}{\beta}})(t-s)^\beta + \sup_{u\in[s,t]}|L_u-L_s| + (Y^n_s - L_s)_-\right) .
\end{align*}

\item\label{item:comparisonii} There exists $C>0$ which depends only on $\sigma,\beta,T$ such that for any $(s,t)\in \mathcal{S}_{[0,T]}$,
\begin{align*}
	|Y^n_t - Y^n_s| \leq C~(C^\RP_J)^2~\left((\vertiii{\RP}_{\beta,[0,T]} \vee \vertiii{\RP}_{\beta,[0,T]}^{\frac{1}{\beta}})(t-s)^\beta + \sup_{u\in[s,t]}|L_u-L_s| + (Y^n_s - L_s)_-\right) .
	\end{align*}
\end{enumerate}
\end{lemma}

\begin{proof}
For any $s\in[0,T)$, define $(\widetilde{Z}^n_{t\leftarrow s})_{t\in[s,T]}$ as the solution of the following (random) ODE:
	\begin{align*}
	\widetilde{Z}^n_{t\leftarrow s} = Z^n_s + C^\RP_J\int_s^t \psi_n\left(\flow[\RP][\widetilde{Z}^n_{u \leftarrow s}][u] -L_u \right) \dd u,~t\in[s,T].
	\end{align*}
	In view of \eqref{eq:boundUJJinv}, there is $W^n(t,z)\leq C^\RP_J \psi_n(\flow[\RP][z][t]-L_t)$, and by the comparison principle for ODEs, it follows that $\widetilde{Z}^n_{t\leftarrow s} \geq Z^n_t$. Observing that 
	\begin{align*}
	\flow[\RP][\widetilde{Z}^n_{u\leftarrow s}][u] = \flow[\RP][Z^n_s][u] +\flow[\RP][\widetilde{Z}^n_{u\leftarrow s}][u]-\flow[\RP][Z^n_s][u] &= \flow[\RP][Z^n_s][u] + \int_{Z^n_s}^{\widetilde{Z}^n_{u\leftarrow s}} \Jac[\RP][z][u] \dd z \\
	&\geq \flow[\RP][Z^n_s][u] + (C^\RP_J)^{-1} (\widetilde{Z}^n_{u\leftarrow s}-Z^n_s) ,
	\end{align*}
	where the last inequality follows from \eqref{eq:lowBoundJinv}, it comes that
	\begin{align*}
	\widetilde{Z}^n_{t\leftarrow s} \leq Z^n_s + C^\RP_J\int_s^t \psi_n\left( \flow[\RP][Z^n_s][u] + (C^\RP_J)^{-1} (\widetilde{Z}^n_{u\leftarrow s}-Z^n_s)  -L_u \right) \dd u,~t\in[s,T].
	\end{align*}
	As the solution of an RDE, there exists $C>0$ depending only on $\beta$ such that $\flow[\RP][y_0][s]$ satisfies (see \cite[Proposition 8.3]{FrizHairer}):
	\begin{align}\label{eq:boundRDEs}
	\|\flow[\RP][y_0][\cdot]\|_{\beta,[0,T]} \leq C \left\{ \left(\|\sigma\|_{\Cb^2} \vertiii{\RP}_{\beta,[0,T]}\right) \vee \left(\|\sigma\|_{\Cb^2} \vertiii{\RP}_{\beta,[0,T]}\right)^{\frac{1}{\beta}} \right\} =: C_{\sigma,\RP,\beta}.
	\end{align}
	It follows that for any $t\in[s,T]$,
	\begin{align*}
	\widetilde{Z}^n_{t\leftarrow s} &\leq Z^n_s + C^\RP_J\int_s^t \psi_n\left(Y^n_s -C_{\sigma,\RP,\beta} (u-s)^\beta + (C^\RP_J)^{-1} (\widetilde{Z}^n_{u \leftarrow s}-Z^n_s) -L_u \right) \dd u \\
	&\leq Z^n_s + C^\RP_J\int_s^t \psi_n\left(-(Y^n_s - L_s)_- -C_{\sigma,\RP,\beta}  (u-s)^\beta + (C^\RP_J)^{-1} (\widetilde{Z}^n_{u \leftarrow s}-Z^n_s) -(L_u-L_s)  \right) \dd u.
	\end{align*}
	For each $s\in[0,T)$, consider now the solution $\overline{Z}^n_{\cdot\leftarrow s}$ of the ODE 
	\begin{align*}
	\overline{Z}^n_{t\leftarrow s} = -(Y^n_s - L_s)_- -C_{\sigma,\RP,\beta} (t-s)^\beta + \int_s^t \psi_n\left( \overline{Z}^n_{u\leftarrow s} -(L_u-L_s) \right) \dd u,~t\in[s,T] .
	\end{align*}
	It satisfies, for all $t\in[s,T]$, $\overline{Z}^n_{t\leftarrow s}\geq -(Y^n_s - L_s)_- -C_{\sigma,\RP,\beta} (t-s)^\beta + (C^\RP_J)^{-1} (\widetilde{Z}^n_{t\leftarrow s}-Z^n_s)$ (by the comparison principle of ODEs). By Lemma~\ref{lem:deterministic}, $\overline{Z}^n_{\cdot\leftarrow s}$ satisfies:
	\begin{align*}
	|\overline{Z}^n_{t\leftarrow s}+(Y^n_s - L_s)_- - (L_t-L_s)| \leq \sqrt{26} \left(C_{\sigma,\RP,\beta} (t-s)^\beta + \sup_{u\in[s,t]}|L_u-L_s| \right), \quad \forall t\in[s,T].
	\end{align*}
	Hence the previous bound yields, for any $s\in[0,T)$ and $t\in[s,T]$,
	\begin{align*}
	Z^n_t \leq \widetilde{Z}^n_{t\leftarrow s} &\leq Z^n_s + C^\RP_J \left( \overline{Z}^n_{t\leftarrow s} +(Y^n_s - L_s)_- +C_{\sigma,\RP,\beta} (t-s)^\beta \right),
	\end{align*}
	and this bound implies easily \ref{item:comparisoni}.

	The second assertion \ref{item:comparisonii} is now derived using the inequality below, \ref{item:comparisoni} and \eqref{eq:boundRDEs}:
	\[
	|Y^n_t - Y^n_s| \leq \sup_{y\in\R}\|\flow[\RP][y][\cdot]\|_{\beta,[0,T]}(t-s)^\beta + C^\RP_J |Z^n_t-Z^n_s| .
	\]
\end{proof}

\begin{remark}
	In the previous proof, we did not estimate directly the Hölder regularity of 
	$\int_0^\cdot \sigma(\flow[\RP][Z^n_s][s]) \dd \RP_s$, since a standard \emph{a priori} estimate would depend on $n$. This question will be treated in the next subsection.
\end{remark}

~

\begin{proposition}\label{lem:convYX} \hspace{2em}
	\begin{enumerate}[label = (\roman*)]
		\item Let the notations and assumptions of Theorem~\ref{th:RfSDE} be in force. Then the sequences of paths $(Z^n)_{n\in\N}$ and $(Y^n)_{n\in\N}$ defined in \eqref{eq:def_XY} are non-decreasing with \(n\). Besides, 
		\begin{equation*}
		\sup_{n\in\N} \sup_{t\in[0,T]} |Z^n_t| < +\infty \quad\text{and} \quad \sup_{n\in\N} \sup_{t\in[0,T]} |Y^n_t| < +\infty .
		\end{equation*}
		\item Now let the assumptions of Theorem~\ref{th:RfSDE_Gaussian} be in force. Then the previous conclusions hold in the almost sure sense and moreover, for any $\gamma\geq 1$,
		\begin{equation*}
		\EE\left[\sup_{n\in\N} \sup_{t\in[0,T]} |Z^n_t|^\gamma\right] < +\infty \quad\text{and} \quad \EE\left[\sup_{n\in\N} \sup_{t\in[0,T]} |Y^n_t|^\gamma\right] < +\infty .
		\end{equation*}
	\end{enumerate}
\end{proposition}

\begin{proof}
	$(i)$ In view of \eqref{eq:lowBoundJinv} and the fact that $\psi_n\leq \psi_{n+1}$, it follows from the comparison theorem for ODEs that $Z^n\leq Z^{n+1}$. Besides, the mapping $z\mapsto \flow[\RP][z]$ is increasing since its derivative is $\Jac[\RP][z]$ which, similarly to \eqref{eq:lowBoundJinv}, is positive. Hence $Y^n\leq Y^{n+1}$.

	Moreover, applying Lemma~\ref{lem:comparison}\ref{item:comparisoni} with $s=0$ yields 
	\begin{align}\label{eq:boundZn}
	Z^n_t \leq y_0 + C~C^\RP_J~\left((\vertiii{\RP}_{\beta,[0,T]} \vee \vertiii{\RP}_{\beta,[0,T]}^{\frac{1}{\beta}})t^\beta + \sup_{u\in[0,t]}|L_u-L_0| \right) ,\quad t\in[0,T].
	\end{align}
Besides, since $Z^n$ is non-decreasing, it follows that \(Z^n_t\geq y_0\), hence $\sup_{n\in\N} \sup_{t\in[0,T]} |Z^n_t| < +\infty$. To prove the second part of claim \emph{(i)}, we apply Lemma~\ref{lem:comparison}\ref{item:comparisonii} with $s=0$ to obtain
	\begin{align}\label{eq:boundYn}
	|Y^n_t| 
	&\leq |y_0| + C~(C^\RP_J)^2~\left((\vertiii{\RP}_{\beta,[0,T]} \vee \vertiii{\RP}_{\beta,[0,T]}^{\frac{1}{\beta}}) t^\beta + \sup_{u\in[0,t]}|L_u-L_0|\right) .
	\end{align}
	
	$(ii)$ Now if $X$ is a Gaussian process satisfying the assumptions of Theorem~\ref{th:RfSDE_Gaussian}, it suffices to use \eqref{eq:boundZn} and \eqref{eq:boundYn}, as well as the following estimates: for any $\gamma\geq 1$,
	\begin{align}\label{eq:boundsGaussL}
	&\EE\left[ \vertiii{\RP}_{\beta,[0,T]}^\gamma \right] <\infty ,\quad \EE\left[ (C^\RP_J)^\gamma \right]<\infty ~~\text{ and }~~\EE\left[\|L\|_{\beta,[0,T]}^\gamma \right]<\infty
	\end{align}
	where the first bound is a classical consequence of Kolmogorov's continuity theorem (which follows from \eqref{hyp:XGauss} for any $\beta  <\tfrac{1}{2r}$), the second one is~\cite[Theorem 6.5]{CassLittererLyons} and the third one was an assumption in Theorem~\ref{th:RfSDE_Gaussian}. Then Claim \emph{(ii)} holds true.
\end{proof}

\subsection{Uniform (in $n$) \emph{a priori} estimates on the sequence of penalised processes}

For any $p\geq 1$ and any $x\geq 0$, recall that $\phi_p(x) = x\vee x^p$, and define the control functions
\begin{align*}
\forall (s,t)\in \mathcal{S}_{[0,T]},\quad &\kappab_{\RP}(s,t):= \vertiii{\RP}_{\pv,[s,t]}^p ,\\
&\kappab_{\RP,K^n}(s,t) := \kappab_{\RP}(s,t) + \left(\delta K^n_{s,t}\right)^p ,
\end{align*}
where
\begin{align}
K^n_t := \int_0^t \psi_n(Y^n_s-L_s) \dd s,\quad \forall t\in[0,T],
\end{align}
denotes the penalisation term in \eqref{eq:penalized_smoothed_fSDE}.

We now obtain estimates on $Y^n$ and $\Rsn$ which are uniform in $n$. These will be crucial in the limiting procedure that leads to the proof of Theorem~\ref{th:RfSDE}.

\begin{lemma}\label{lem:unifBoundRYn}
	$(i)$ Under the assumptions of Theorem~\ref{th:RfSDE}, one has
	\begin{align}\label{eq:boundYnRYn}
	\Theta := \sup_{n\in\N} \left( \|\sigma'(Y^n)\sigma(Y^n)\|_{\pv,[0,T]} + \|\Rsn\|_{\frac{p}{2},[0,T]} \right) <\infty .
	\end{align}
	$(ii)$ If in addition, the assumptions of Theorem~\ref{th:RfSDE_Gaussian} hold, then $\EE\left( \Theta^\gamma \right)<\infty$, for any $\gamma\geq 1$.
\end{lemma}

\begin{proof}
This proof relies strongly on the \emph{a priori} estimates obtained in Lemma~\ref{lem:RDeltaSig}. Notice that this lemma applies here since we have established in Proposition~\ref{prop:globalExistence} that $Y^n$ is controlled by $X$. Hence by Lemma~\ref{lem:RDeltaSig}\ref{item:i}, one gets that for any $(s,t)\in\mathcal{S}_{[0,T]}$,
	\begin{align}\label{eq:boundYnpv}
	\|Y^n\|_{p,[s,t]}^p \leq C ~\phi_p\left(\kappab_{\RP,K^n}(s,t) \right) ,
	\end{align}
	for some $C>0$ that depends only on $p$ and $\sigma$.

Now from Lemma~\ref{lem:RDeltaSig}\ref{item:i'}, one gets there exists $\delta_{\RP}>0$ with the property that for any $(s,t)\in\mathcal{S}_{[0,T]}$ such that $|t-s|\leq \delta_\RP$, 
	\begin{align*}
	\|\Rsn\|_{\ppv,[s,t]}^\ppv \leq C ~\kappab_{\RP,K^n}(s,t).
	\end{align*}
	Then for any $(s,t)\in\mathcal{S}_{[0,T]}$, 
	\begin{align}\label{eq:decompRsn}
	\|\Rsn\|_{\frac{p}{2},[s,t]}^{\frac{p}{2}} &= \sup_{\pi = (t_i)\subset [s,t]} \left(\sum_{t_{i+1}-t_i\leq \delta_\RP} |\Rsn_{t_i, t_{i+1}}|^{\frac{p}{2}} + \sum_{t_{i+1}-t_i> \delta_\RP} |\Rsn_{t_i, t_{i+1}}|^{\frac{p}{2}}  \right) \nonumber\\
	&\leq C ~\kappab_{\RP,K^n}(s,t) + \sup_{\pi = (t_i)\subset[s,t]} \sum_{t_{i+1}-t_i> \delta_\RP} |\Rsn_{t_i, t_{i+1}}|^{\frac{p}{2}} .
	\end{align}
	By a standard rough path procedure, one gets that
	\begin{align*}
	\sum_{t_{i+1}-t_i> \delta_\RP} |\Rsn_{t_i, t_{i+1}}|^{\frac{p}{2}} \leq C \left(\frac{t-s}{\delta_\RP}\right)^{\frac{3p}{2}-2}\left( \kappab_{\RP,K^n}(s,t) +\|X\|_{p,[s,t]}^p\right),
	\end{align*}
	so that by using the above inequality in \eqref{eq:decompRsn}, it follows that
	\begin{align}\label{eq:boundRsigman0}
	\|\Rsn\|_{\frac{p}{2},[s,t]}^{\frac{p}{2}} \leq C \left(1+ \left(\frac{t-s}{\delta_\RP}\right)^{\frac{3p}{2}-2} \right) \kappab_{\RP,K^n}(s,t)  ,\quad \forall (s,t)\in\mathcal{S}_{[0,T]}.
	\end{align}
	In view of the definition \eqref{eq:def_XY} of $Z^n$ and the bound \eqref{eq:boundUJJinv} on $J$ (recall also that $J$ is positive), one has
	\begin{align*}
	K^n_t - K^n_s &= \int_s^t \Jac[\RP][Z^n_u][u] \Jac[\RP][Z^n_u][0][u] \psi_n(Y^n_u - L_u) \dd u \leq C^\RP_J \left( Z^n_t-Z^n_s \right).
	\end{align*}
	Hence $\kappab_{\RP,K^n}(s,t) \leq (C^\RP_J \delta Z^n_{s,t})^p + \vertiii{\RP}_{\ppv,[s,t]}^p$ and since 
	$Z^n\leq Z^{n+1}$ (and $Z^n_0 = Z^{n+1}_0=y_0$), 
	\begin{align}\label{eq:boundRsigman}
	\|\Rsn\|_{\frac{p}{2},[0,T]}^{\frac{p}{2}} \leq C \left(1+\delta_\RP^{2-\frac{3p}{2}}\right) \left( (C^\RP_J \delta Z^n_{s,t})^p + \vertiii{\RP}_{\ppv,[s,t]}^p\right) .
	\end{align}
	Recall from Lemma~\ref{lem:RDeltaSig} that $\delta_\RP\geq C^{-1}\vertiii{\RP}_\beta^{-1/\beta} \wedge T>0$. Hence $(i)$ is proven.
	
	To prove $(ii)$, note that \eqref{eq:boundYnpv} and the observation of the previous paragraph imply that
	\begin{align*}
	\|Y^n\|_{p,[s,t]} \leq C \left( \phi_p\left(C^\RP_J Z_T\right) + \phi_p\left(\vertiii{\RP}_{\pv,[0,T]}\right)^{\frac{1}{p}}\right) .
	\end{align*}
	Hence, combined with inequalities \eqref{eq:boundZn} and \eqref{eq:boundsGaussL}, one gets that $\EE\left[ \sup_{n\in\N} \|Y^n\|_{\pv,[0,T]}^\gamma \right]<\infty$, $\forall \gamma\geq 1$. Using the regularity of $\sigma$, we deduce that $\EE\left[ \sup_{n\in\N}\|\sigma'(Y^n)\sigma(Y^n)\|_{\pv,[0,T]}^\gamma \right]<\infty$. 
	As ${\delta_\RP \geq C^{-1} \vertiii{\RP}_{\beta}^{-1/\beta} \wedge T}$, it follows that for any $\gamma\geq 1$,
	\begin{align}\label{eq:boundDeltaX}
	\EE[\delta_\RP^{(2-\frac{3p}{2})\gamma}]\leq C \EE\left[\left(\vertiii{\RP}_\beta^{\frac{1}{\beta}} \vee T^{-1}\right)^{(\frac{3p}{2}-2)\gamma}\right]<\infty . 
	\end{align}
	In view of \eqref{eq:boundRsigman} and the previous inequality, it follows that $\EE\left[ \sup_{n\in\N}\|\Rsn\|_{\pv,[0,T]}^\gamma \right]<\infty$.
\end{proof}

\subsection{Uniform (in $n$) continuity of the sequence of penalised processes}\label{subsec:unifcont}

\begin{corollary}\label{cor:unifHolder}
	Recall that $\beta = \tfrac{1}{p}$. The rough integral in \eqref{eq:penalized_smoothed_fSDE} is $\beta$-Hölder continuous on $[0,T]$, uniformly in $n$: there exists $C>0$ depending only on $\beta$, $T$, $\|\sigma\|_{\infty}$ and $\|\sigma'\|_\infty$ such that
	\begin{align*}
	\forall n\in \N,~\forall (s,t)\in\mathcal{S}_{[0,T]},\quad \big|\int_s^t \sigma(Y^n_u) \dd\RP_u \big| \leq C(1+\Theta) \left(  \|X\|_\beta + \|\mathbb{X}\|_{2\beta} \right)  |t-s|^\beta,
	\end{align*}
	where $\Theta$ was defined (independently of $n$) in \eqref{eq:boundYnRYn}. 
\end{corollary}

\begin{proof}
	In view of Proposition~\ref{prop:globalExistence} and Lemma~\ref{lem:unifBoundRYn}, Theorem~\ref{th:RI} implies that $\| \int_0^\cdot \sigma(Y^n_u) \dd \RP_u\|_{\pv,[s,t]}$ is bounded from above by some control function which is independent of $n$:
	\begin{align*}
	\| \int_0^\cdot \sigma(Y^n_u) \dd \RP_u\|_{\pv,[s,t]} &\leq \|\sigma(Y^n)\|_{\infty,[s,t]}\|X\|_{p,[s,t]} + \|\sigma'(Y^n)\sigma(Y^n)\|_{\infty,[s,t]} \|\mathbb{X}\|_{\ppv,[s,t]} \\
	&\quad + C_p\left(\|X\|_{p,[s,t]} \|\Rsn\|_{\frac{p}{2},[s,t]} + \|\mathbb{X}\|_{\ppv,[s,t]} \|\sigma'(Y^n)\sigma(Y^n)\|_{p,[s,t]}  \right) \\
	&\leq C \left(1+\Theta \right) \left( \|X\|_{p,[s,t]} + \|\mathbb{X}\|_{\frac{p}{2},[s,t]} \right) .
	\end{align*}
	To conclude the proof, it remains to notice that since $(X,\mathbb{X}) \in \Ccalbeta_g$, $\|X\|_{p,[s,t]} \leq \|X\|_\beta |t-s|^\beta$ and $\|\mathbb{X}\|_{\frac{p}{2},[s,t]} \leq \|\mathbb{X}\|_{2\beta} |t-s|^{2\beta} \leq T^\beta \|\mathbb{X}\|_{2\beta} (t-s)^\beta$.
\end{proof}

\begin{proposition}\label{prop:conv_neg_part}
	$(i)$ Under the assumptions of Theorem~\ref{th:RfSDE}, there exists $C>0$ which depends only on $p$ and $T$, such that for any $n\in\N^*$, 
	\[ \sup_{s\in[0,T]} (Y^n_s-L_s)_- \leq C\left(1+ (1+\Theta)\left(  \|X\|_\beta + \|\mathbb{X}\|_{2\beta} \right) + \|L\|_\beta \right) n^{-\beta},
	\]
	where $\Theta$ was defined in \eqref{eq:boundYnRYn}. In particular, $\displaystyle \lim_{n\rightarrow \infty} \sup_{s\in[0,T]} (Y^n_s-L_s)_- = 0$.\\
	$(ii)$ If in addition, the assumptions of Theorem~\ref{th:RfSDE_Gaussian} hold, then ${\displaystyle \lim_{n\rightarrow \infty} \EE\big[\sup_{s\in[0,T]} |(Y^n_s-L_s)_-|^\gamma \big] = 0}$, $\forall\gamma\geq 1$.
\end{proposition}

\begin{proof}
	$(i)$ Recall that $(x)_- \leq \tfrac{1}{n}\psi_n(x)+\tfrac{1}{2n}$. Applying Lemma~\ref{lem:deterministic}$(ii)$ and Corollary~\ref{cor:unifHolder}, one gets that
	\begin{align}\label{eq:boundNegPart}
	\forall n,\quad \sup_{s\in[0,T]} (Y^n_s-L_s)_- &\leq C \left(1 + \|L\|_\beta + \|\int_0^\cdot \sigma(Y^n_u) \dd\RP_u\|_\beta \right) n^{-\beta} +\frac{1}{2n}\nonumber\\
	&\leq C \left(1 + \|L\|_\beta + (1+\Theta)\left( \|X\|_\beta + \|\mathbb{X}\|_{2\beta} \right) \right) n^{-\beta} .
	\end{align}
	which is the desired result.
	
	$(ii)$ Using \eqref{eq:boundsGaussL} and Lemma~\ref{lem:unifBoundRYn} $(ii)$, one gets that $\EE[\left((1+\Theta)\left(  \|X\|_\beta + \|\mathbb{X}\|_{2\beta} \right)\right)^\gamma] <\infty$, $\forall \gamma\geq 1$, so the result follows from \eqref{eq:boundNegPart}.
\end{proof}

\begin{proposition}\label{prop:unifConvYn}
	$(i)$ Under the assumptions of Theorem~\ref{th:RfSDE}, $(Y^n)_{n\in\N}$ converges uniformly to some process $(Y_t)_{t\in[0,T]} \in \Cbeta$.\\ 
	$(ii)$ If in addition, the driving noise $X$ is Gaussian and the assumptions of Theorem~\ref{th:RfSDE_Gaussian} are satisfied, then the convergence happens in $L^\gamma\left(\Omega; (\mathcal{C}^0,\|\cdot\|_{\infty,[0,T]})\right)$ (i.e. as in \eqref{eq:convYnGauss}), for any $\gamma\geq 1$. Besides, $Y$ has a $\beta$-Hölder continuous modification and $\EE[\|Y\|_{\beta,[0,T]}^\gamma]<\infty$, for any $\gamma\geq 1$.
\end{proposition}

\begin{proof}
	$(i)$ 
	In view of the two inequalities of Lemma~\ref{lem:comparison}, we can now use Proposition~\ref{prop:conv_neg_part} to pass to the limit in the right-hand side of these inequalities, 
	to get that $Z$ and $Y$ are (Hölder-)continuous. Hence, arguing with Dini's Theorem, we now conclude that the convergences are uniform. \\
	$(ii)$ Under the assumptions of Theorem~\ref{th:RfSDE_Gaussian}, one deduces from the previous point that  almost surely, ${\lim_{n\rightarrow\infty}\|Y^n - Y\|_{\infty,[0,T]} = 0}$. Moreover, Proposition~\ref{lem:convYX} states that $(Y^n)_{n\in\N}$ is a non-decreasing sequence. Thus $ \|Y^n - Y\|_{\infty,[0,T]} \leq 2\|Y\|_{\infty,[0,T]}$, and since $\EE[\|Y\|_{\infty,[0,T]}^\gamma] <\infty$ (by Proposition~\ref{lem:convYX} $(ii)$), the convergence result is obtained by using Lebesgue's theorem.\\
	The Hölder continuity of $Y$ is a consequence of Lemma~\ref{lem:comparison}\ref{item:comparisonii} and Proposition~\ref{prop:conv_neg_part} $(ii)$.
\end{proof}


\subsection{Identification of the limit process $Y$}\label{subsec:identification}

\noindent\textbf{Step 1: convergence of the rough integral.}

\begin{proposition}\label{prop:convInt}
	The path $\sigma(Y)$ is controlled by $X$, more precisely $(\sigma(Y),\sigma'(Y)\sigma(Y))\in \Vp_X$, and the following convergence happens in $\mathcal{C}^0([0,T],\R)$:
	\begin{align*}
	\lim_{n\rightarrow \infty} \left\|\int_0^\cdot \sigma(Y^n_s)\dd \RP_s - \int_0^\cdot \sigma(Y_s)\dd \RP_s \right\|_{\infty,[0,T]} = 0.
	\end{align*}
\end{proposition}

\begin{proof}
First recall that $\sigma(Y^n)$ is controlled by $X$ and that its Gubinelli derivative is $\sigma'(Y^n) \sigma(Y^n)$ (Remark~\ref{rem:sYn}). 
By Proposition~\ref{prop:unifConvYn}, $\sigma'(Y)\sigma(Y)\in \Cbeta$, $\sigma'(Y^n) \sigma(Y^n)$ converges uniformly to $\sigma'(Y)\sigma(Y)$, and in view of Lemma~\ref{lem:unifBoundRYn}, it follows that $\|\sigma'(Y^n) \sigma(Y^n)-\sigma'(Y)\sigma(Y)\|_{p',[0,T]}\rightarrow 0$ for any $p'>p$ (by \cite[Lemma 5.27]{FrizVictoir}). Similarly, $\Rsn_{s,t}$ converges uniformly to $R^{\sigma(Y)}_{s,t} := \delta\sigma(Y)_{s,t} - \sigma'(Y_s) \sigma(Y_s)\delta X_{s,t}$, and it follows from Lemma~\ref{lem:unifBoundRYn} that $R^{\sigma(Y)} \in \mathcal{V}^{\frac{p}{2}}$ (by \cite[Lemma 5.12]{FrizVictoir}) and $\|R^{\sigma(Y^n)}-R^{\sigma(Y)}\|_{\frac{p'}{2},[0,T]}\to 0$, for any $p'>p$. Hence $(Y,Y') \in \Vp_X$. Due to Inequality \eqref{eq:boundRI} applied with $p'>p$ and the two previous convergences in $p'$-variation, $\|\int_0^\cdot \sigma(Y^n_s)\dd \RP_s - \int_0^\cdot \sigma(Y_s)\dd \RP_s \|_{p',[0,T]} \to 0$, so the result follows in uniform norm.
\end{proof}

A direct consequence of the previous proposition and of Proposition~\ref{prop:unifConvYn} is that $K^n$ converges (uniformly) to a limit path $K$ so that for any $t\in[0,T]$, $Y_t = y_0 + \int_0^t \sigma(Y_s) \dd\RP_s + K_t$. As a limit of non-decreasing paths, $K$ is non-decreasing. Hence the properties \ref{def:SP1} and \ref{def:SP3} of Definition~\ref{def:SP} are verified. \\

\noindent \textbf{Step 2: $Y\geq L$.} This is the result of Proposition~\ref{prop:conv_neg_part}. Thus property \ref{def:SP2} of Definition~\ref{def:SP} is satisfied.

\noindent\textbf{Step 3: points of increase of $K$.} By the uniform convergence of $K^n$ and the non-decreasing property of $K^n$ and $K$, it follows that $\dd K^n$ weakly converges towards $\dd K$ and since $Y^n$ converges uniformly to $Y$,
\begin{align*}
0\geq \int_0^t (Y^n_s-L_s) \psi_n(Y^n_s-L_s) \dd s= \int_0^t (Y^n_s-L_s) \dd K^n_s \rightarrow \int_0^t (Y_s-L_s) \dd K_s,
\end{align*}
where the last integral exists in the sense of Lebesgue-Stieltjes integrals, since $K$ is a non-decreasing path. Since $Y_s-L_s\geq 0$ (by the previous step) and $K$ is non-decreasing, it follows that $\int_0^t (Y_s-L_s) \dd K_s \geq 0$. Hence for any $t\in[0,T]$, $\int_0^t (Y_s-L_s) \dd K_s= 0$, which proves that the point \ref{def:SP4} is satisfied.

~

\emph{In view of the Steps $1$ to $3$, $(Y,K)$ satisfies the properties \ref{def:SP1}, \ref{def:SP2}, \ref{def:SP3} and \ref{def:SP4} of Definition~\ref{def:SP}. Hence $(Y,K)$ is a solution to $SP(\sigma,\RP,L)$. In addition, if $X$ is a Gaussian process satisfying Assumption \eqref{hyp:XGauss}, we obtained in this section the probabilistic estimates that prove Theorem~\ref{th:RfSDE_Gaussian}.}

\section{Uniqueness in the Skorokhod problem and rate of convergence of the sequence of penalised paths}\label{sec:rate}

\subsection{Uniqueness (Proof of Theorem \ref{th:uniqueness})}\label{subsec:uniqueness}

In the case $\beta>\tfrac{1}{2}$, the uniqueness of the reflected solution is due to \citet{Slominski}. In the case $\beta\leq\tfrac{1}{2}$, the uniqueness of the reflected RDE has been proven recently by \citet{DeyaEtAl}. The difference between our work and \cite{DeyaEtAl} is that they have a fixed boundary process $L\equiv 0$. Thanks to our Lemma~\ref{lem:RDeltaSig}, uniqueness still holds in case of a moving boundary with little change compared to \cite{DeyaEtAl}.

Let $(Y^1,K^1)$ and $(Y^2,K^2)$ be two solutions to $SP(\sigma,\RP,L)$, both with initial condition $y_0\geq L_0$. 
 By definition, $Y^1$ and $Y^2$ are controlled in $p$-variation. Since the main ingredients in this section are Lemma~\ref{lem:RDeltaSig} and the rough Gr\"onwall lemma of \cite{DeyaEtAl0}, which both involve control functions (see also Remark~\ref{rem:pvar}), the $p$-variation topology is a natural choice.
 Consider the operator $\Delta$ which acts on functionals of $Y^1$ and $Y^2$ as follows: for any $\Phi: \mathcal{C}([0,T];\R) \rightarrow \mathcal{C}([0,T];\R)$,
\begin{align*}
\Delta \Phi(Y) = \Phi(Y^1)-\Phi(Y^2).
\end{align*}
For instance, we shall write $\Delta Y_s = Y^1_s-Y^2_s, ~\Delta \sigma(Y)_s = \sigma(Y^1_s) - \sigma(Y^2_s)$ and also $\delta(\Delta Y)_{s,t} = \Delta (\delta Y_{s,t}) = Y^1_t - Y^1_s - Y^2_t + Y^2_s$, etc.\\
By linearity, $\Delta Y$ has Gubinelli derivative $\Delta \sigma(Y)$, i.e. $(\Delta Y,\Delta \sigma(Y)) \in \Vp_X$, and $R^{\Delta Y}_{s,t}:= \delta(\Delta Y)_{s,t} - \Delta\sigma(Y)_s\delta X_{s,t} \in \mathcal{V}^{\frac{p}{2}}$. Similarly, $\Delta \sigma(Y)$ has Gubinelli derivative $\Delta \sigma'\sigma(Y)$. Note that $R^{\Delta Y} = \Delta R^Y$ and that $R^{\Delta\sigma(Y)} = \Delta R^{\sigma(Y)}$.

Set $w_K:\mathcal{S}_{[0,T]} \rightarrow \R_+$, a control function associated to $(K^1,K^2)$, as follows
\begin{align*}
w_K(s,t) = \|(K^1,K^2)\|_{1\text{-var},[s,t]} = \delta K^1_{s,t} + \delta K^2_{s,t},
\end{align*}
and define the control functions $\kappab_{\RP,K}$ and $\widetilde{\kappab}_{\RP,K}$ by
\begin{equation}\label{eq:defKappaXwn}
\begin{split}
\forall (s,t)\in \mathcal{S}_{[0,T]},\quad \kappab_{\RP,K}(s,t) &:= \vertiii{\RP}_{\pv,[s,t]}^p + w_K(s,t)^p \\
\widetilde{\kappab}_{\RP,K}(s,t) &:= \sum_{j=1}^4 \kappab_{\RP,K}(s,t)^j.
\end{split}
\end{equation}

Without loss of generality, assume that there exists $\tau\in(0,T]$ such that for any $t\in[0,\tau]$, $Y^1_t \geq Y^2_t$. Observe that 
	\begin{align}\label{eq:deltaDeltaY}
	\delta(\Delta Y)_{s,t} + w_K(s,t) &= \int_s^t \Delta\sigma(Y)_u \dd\RP_u + \delta(\Delta K)_{s,t} + w_K(s,t) \nonumber\\
	&= \int_s^t \Delta\sigma(Y)_u \dd\RP_u + 2\delta K^1_{s,t} .
	\end{align}
	By applying Inequality \eqref{eq:boundRI}, one gets that for any $s<t\in[0,\tau]$,
	\begin{align*}
	|\int_s^t \Delta\sigma(Y)_u \dd\RP_u| \leq C \left\{ \Delta Y_s |\delta X_{s,t}| + \Delta Y_s |\mathbb{X}_{s,t}| +\|R^{\Delta\sigma(Y)}\|_{\ppv,[s,t]} \|X\|_{p,[s,t]} + \|\Delta\sigma'\sigma(Y)\|_{p,[s,t]}\|\mathbb{X}\|_{\ppv,[s,t]} \right\} .
	\end{align*}
	With $\delta_\RP>0$ as defined in Lemma~\ref{lem:RDeltaSig}, one now applies Lemma~\ref{lem:RDeltaSig}\ref{item:ii} and Inequality \eqref{eq:deltaDeltasig2} to get that for any $(s,t)\in \mathcal{S}_{[0,\tau]}$ such that $|t-s|\leq \delta_\RP$,
	\begin{align}\label{eq:boundIntDelta}
	|\int_s^t \Delta\sigma(Y)_u \dd\RP_u| \leq C \left( \|\Delta Y\|_{\infty,[s,t]} + w_K(s,t) \right) \left\{(1+\|\mathbb{X}\|_{\ppv,[s,t]}) 
	\phi_p\left( \widetilde{\kappab}_{\RP,K}(s,t)^{\frac{1}{p}} \right)
	\right\}.
	\end{align}
	Thus we get from \eqref{eq:deltaDeltaY} and \eqref{eq:boundIntDelta} that
	\begin{align*}
	\delta(\Delta Y)_{s,t} + w_K(s,t) &\leq C \left( \|\Delta Y\|_{\infty,[s,t]} + w_K(s,t) \right) (1+\|\mathbb{X}\|_{\ppv,[s,t]}) 
	\phi_p\left( \widetilde{\kappab}_{\RP,K}(s,t)\right)^{\frac{1}{p}}
	+ 2\delta K^1_{s,t} .
	\end{align*}
	We are now in a position to apply the rough Gr\"onwall lemma of \citet[Lemma 2.11]{DeyaEtAl0}\footnote{ Note that this lemma must apply to a non-negative function (here $\Delta Y$). In \cite{DeyaEtAl0}, the authors consider the function $|\Delta Y|$, but P. Gassiat brought to our attention
 that it suffices to assume $Y^1\geq Y^2$.} which reads:
	 for any $(s,t)\in \mathcal{S}_{[0,\tau]}$ such that $|t-s|\leq \delta_\RP$,
	\begin{align*}
	\|\Delta Y\|_{\infty,[s,t]} + w_K(s,t) \leq 2 \exp\left\{1\vee \left( C (1+\|\mathbb{X}\|_{\ppv,[s,t]})^p 
	\phi_p\left( \widetilde{\kappab}_{\RP,K}(s,t) \right) 
	\right)\right\} \left(\Delta Y_s + 2\delta K^1_{s,t}\right) .
	\end{align*}
	Assume now that there exists a sub-interval $(s,t) \subset [0,\tau]$ of length at most $\delta_\RP$ such that $\Delta Y_u>0$ for any $u\in(s,t)$ and $\Delta Y_s=0$. In that case, $Y^1_u>L_u$ and it follows that $\delta K^1_{s,u}=0$ for any $u\in(s,t)$. Hence $\|\Delta Y\|_{\infty,[s,t]} + w_K(s,t) \leq 0$ and this contradicts the existence of $s$ and $t$. Thus $\Delta Y=0$, and there is at most one solution to $SP(\sigma,\RP,L)$.

\subsection{Rate of convergence (Proof of Theorems \ref{th:rate} and \ref{th:rateGaussian})}

With the notations of the previous subsection, $Y^1$ now corresponds to $Y$ (the reflected path) and $Y^2$ corresponds to the penalised path $Y^n$. The operator $\Delta$ will now be denoted $\Delta_{n}$. For instance, recall that $R^{\Delta_n \sigma(Y)}_{s,t}:= \delta(\Delta_n \sigma(Y))_{s,t} - \Delta_n\sigma'\sigma(Y)_s\delta X_{s,t} \in \mathcal{V}^{\frac{p}{2}}$. 
Here, the control function associated to $(K,K^n)$ will be denoted by $w_n$ (instead of $w_K$ in Subsection~\ref{subsec:uniqueness}), 
while $\kappab_{\RP,w_n}$ and $\widetilde{\kappab}_{\RP,w_n}$ denote the control functions which correspond respectively to $\kappab_{\RP,K}$ and $\widetilde{\kappab}_{\RP,K}$ in Subsection~\ref{subsec:uniqueness}.

Denote by $\Pi_t(y)$ the projection on the epigraph of
\(L\), i.e. $\Pi_t(y)=y\vee L_t$.
Consider the control function defined by: $\forall (s,t)\in \mathcal{S}_{[0,T]}$,
\begin{align}\label{eq:defKappaBar}
\overline{\kappab}_{\RP,w_n}(s,t) := (1+\|\mathbb{X}\|_{\ppv,[s,t]})^p 
\phi_p\left( \widetilde{\kappab}_{\RP,w_n}(s,t)\right) .
\end{align}
Before stating a first technical result, we recall that in Lemma~\ref{lem:RDeltaSig}, it is proven that there exists a time $\delta_\RP>0$, and that for any $(s,t)\in \mathcal{S}_{[0,T]}$ such that $|t-s|\leq \delta_\RP$,
\begin{align}\label{eq:appliLemRDeltaSig}
	\|R^{\Delta_n\sigma(Y)}\|_{\ppv,[s,t]} \leq C \left\{w_n(s,t) + (w_n(s,t)+\|\Delta_n Y\|_{\infty,[s,t]}) 
	\phi_p\left(\widetilde{\kappab}_{\RP,w_n}(s,t)^{\frac{1}{p}} \right)
	\right\} .
\end{align}

\begin{lemma}\label{lem:bound1sec5}
	Denote by $\overline{\Theta}$ the quantity $C\left(1 +(1+\Theta)\left(\|X\|_\beta + \|\mathbb{X}\|_{2\beta} \right) + \|L\|_\beta\right)$ that appears in the upper bound of $\sup_{s\in[0,T]}(Y^n_s-L_s)_-$ in Proposition~\ref{prop:conv_neg_part}. There exists $C>0$ such that for any $n\in\N^*$, for any $(s,t)\in \mathcal{S}_{[0,T]}$ such that $|t-s|\leq \delta_\RP$, one has
	\begin{align*}
	\|Y-\Pi_\cdot(Y^n)\|_{\infty,[s,t]} &\leq 2 e^{1\vee \left( C \overline{\kappab}_{\RP,w_n}(s,t) \right)} \left( Y_s-\Pi_s(Y^n_s) + \overline{\Theta} n^{-\beta}  + 2\delta K_{s,t}\right) .
	\end{align*}
\end{lemma}

\begin{proof}
	Mimicking the beginning of the proof of the previous subsection (in particular Eq. \eqref{eq:deltaDeltaY} and \eqref{eq:boundIntDelta}), then applying Lemma~\ref{lem:RDeltaSig}\ref{item:ii} (which is exactly Equation \eqref{eq:appliLemRDeltaSig}) and Inequality \eqref{eq:deltaDeltasig2}, one gets that for any $(s,t)\in \mathcal{S}_{[0,T]}$ such that $|t-s|\leq \delta_\RP$,
	\begin{align*}
	|\int_s^t \Delta_n\sigma(Y)_u \dd\RP_u| \leq C \left( \|\Delta_n Y\|_{\infty,[s,t]} + w_n(s,t) \right) \left\{(1+\|\mathbb{X}\|_{\ppv,[s,t]}) 
	\phi_p\left( \widetilde{\kappab}_{\RP,w_n}(s,t)^{\frac{1}{p}} \right) 
	\right\}.
	\end{align*}
	and then 
	\begin{align*}
	\delta(\Delta_n Y)_{s,t} + w_n(s,t) &\leq C\left( \|\Delta_n Y\|_{\infty,[s,t]} + w_n(s,t) \right) \overline{\kappab}_{\RP,w_n}(s,t)^{\frac{1}{p}}+ 2\delta K_{s,t} .
	\end{align*}
	The rough Gr\"onwall lemma of \citet[Lemma 2.11]{DeyaEtAl0} now reads: $\forall (s,t)\in \mathcal{S}_{[0,T]}$ such that $|t-s|\leq \delta_\RP$,
	\begin{align*}
	\|\Delta_n Y\|_{\infty,[s,t]} + w_n(s,t) \leq 2 \exp\left\{1\vee \left( C \overline{\kappab}_{\RP,w_n}(s,t) \right)\right\} \left(\Delta_n Y_s + 2\delta K_{s,t}\right) .
	\end{align*}
	With the current notation, Proposition~\ref{prop:conv_neg_part} yields $\Pi_t(Y^n_t) -Y^n_t \leq \overline{\Theta} n^{-\beta}$, $\forall t\in[0,T]$. Since $Y-\Pi_\cdot(Y^n) \leq \Delta_n Y$, we obtain that for any $(s,t)\in\mathcal{S}_{[0,T]}$ such that $|t-s|\leq \delta_\RP$,
	\begin{align*}
	 \|Y-\Pi_\cdot(Y^n)\|_{\infty,[s,t]} &\leq 2 e^{1\vee \left( C \overline{\kappab}_{\RP,w_n}(s,t) \right)} \left( Y_s-\Pi_s(Y^n_s) + \Pi_s(Y^n_s)-Y^n_s  + 2\delta K_{s,t}\right) ,
	\end{align*}
	which gives the expected result.
\end{proof}

We now have all the ingredients to carry out the proofs of Theorems~\ref{th:rate} and \ref{th:rateGaussian}.
\begin{proof}[Proof of Theorem~\ref{th:rate}]
	Let us admit the following inequality, that will be proven in the second part of this proof
	\begin{equation}\label{eq:boundRate}
	\|Y-\Pi_\cdot(Y^n)\|_{\infty,[0,T]} \leq \exp\left\{2T\delta_\RP^{-1} + \left(C \overline{\kappab}_{\RP,w_n}(0,T) \right)\right\} (2 + T\delta_\RP^{-1} \overline{\Theta}) n^{-\beta}. 
	\end{equation}
	It is clear that the quantity $\overline{\kappab}_{\RP,K}(0,T) := \sup_{n\in\N}\overline{\kappab}_{\RP,w_n}(0,T)$ is finite (we give more details in the next proof, where $\RP$ is a Gaussian rough paths). Hence the Inequality \eqref{eq:boundRate} yields the desired result since
	\begin{align}\label{eq:finalBound}
	\|\Delta_n Y\|_{\infty,[0,T]} &\leq \|Y-\Pi_\cdot(Y^n)\|_{\infty,[0,T]} +\|\Pi_\cdot(Y^n)-Y^n\|_{\infty,[0,T]} \nonumber\\
	&\leq \exp\left\{2T\delta_\RP^{-1} + \left(C \overline{\kappab}_{\RP,K}(0,T) \right)\right\} (2 +\overline{\Theta} + T\delta_\RP^{-1} \overline{\Theta}) n^{-\beta} ,
	\end{align}
	using Proposition~\ref{prop:conv_neg_part} in the last inequality.
	
		Let us now prove \eqref{eq:boundRate}.
	Consider first the case $s=0$ and $t\in (0,\delta_\RP)$: since $Y_0 = \Pi_0(Y^n_0) = y_0$, one can define 
	\begin{align*}
	t_0^n := T\wedge \inf\left\{t>0:~Y_{t} -\Pi_{t}(Y^n_{t}) = 2n^{-\beta}  \right\} .
	\end{align*}
	Of course, if $t_0^n = T$ then the proof is over. So let us assume that $t_0^n<T$ and define 
	\begin{align*}
	t_1^n := T\wedge \inf\left\{t>t_0^n:~Y_{t} -\Pi_{t}(Y^n_{t}) = n^{-\beta}  \right\} 
	\end{align*}
	and the mapping $\vartheta: [0,T) \rightarrow [0,T]$ associated to $\delta_{\RP}$ as follows:
	\begin{align*}
	\forall t\in [0,T), \quad \vartheta(t) = T\wedge (t+\delta_\RP).
	\end{align*}
	Notice that for any $u\in [t_0^n,t_1^n]$, $Y_u$ lies strictly above the boundary $L_u$ since $Y_u\geq n^{-\beta} + \Pi_u(Y^n_u)>L_u$. Hence  $\delta K_{t_0^n,t_1^n}=0$. It follows  from Lemma~\ref{lem:bound1sec5} that for any $u\in [t_0^n,t_1^n\wedge \vartheta(t_0^n)]$,
	\begin{align}\label{eq:bound0Y-PiYn}
	\|Y-\Pi_\cdot(Y^n)\|_{\infty,[t_0^n,u]} \leq 2e^{1\vee \left( C \overline{\kappab}_{\RP,w_n}(t_0^n,t_1^n\wedge \vartheta(t_0^n)) \right)}\left(2n^{-\beta} + \overline{\Theta} n^{-\beta} +0  \right) .
	\end{align}
	Then, we have one of the following possibilities:
	\begin{enumerate}
		\item if $t_1^n\wedge \vartheta(t_0^n)=T$, then rephrasing \eqref{eq:bound0Y-PiYn}, we get ${\|Y-\Pi_\cdot(Y^n)\|_{\infty,[t_0^n,T]} \leq 2e^{1\vee \left( C \overline{\kappab}_{\RP,w_n}(t_0^n,T) \right)} (2 + \overline{\Theta}) n^{-\beta}}$ and \eqref{eq:boundRate} is proven (recall that by definition, $\|Y-\Pi_\cdot(Y^n)\|_{\infty,[0, t_0^n]} \leq 2n^{-\beta}$);
		\item if $t_1^n\wedge \vartheta(t_0^n)=t_1^n <T$, then \eqref{eq:bound0Y-PiYn} now reads $${\|Y-\Pi_\cdot(Y^n)\|_{\infty,[t_0^n,t_1^n]} \leq 2e^{1\vee \left( C \overline{\kappab}_{\RP,w_n}(t_0^n,t_1^n) \right)}(2 + \overline{\Theta}) n^{-\beta}},$$
		which is smaller than the right-hand side of inequality \eqref{eq:boundRate}. Then, since $Y_{t_1^n} -\Pi_{t_1^n}(Y^n_{t_1^n}) = n^{-\beta}$, we can define 
		\begin{align*}
		t_{2}^n := T\wedge \inf\left\{t>t_1^n:~Y_{t} -\Pi_{t}(Y^n_{t}) = 2n^{-\beta}  \right\}
		\end{align*}
		and if $t_{2}^n < T$, define also
		\begin{align*}
		t_{3}^n := T\wedge \inf\left\{t>t_{2}^n:~Y_{t} -\Pi_{t}(Y^n_{t}) = n^{-\beta}  \right\} 
		\end{align*}
		and move to the next step (iterate);
		
		\item if $t_1^n\wedge \vartheta(t_0^n)=\vartheta(t_0^n) <T$, then as in \eqref{eq:bound0Y-PiYn}, we get $\|Y-\Pi_\cdot(Y^n)\|_{\infty,[t_0^n,\vartheta(t_0^n)]} \leq 2e^{1\vee \left( C \overline{\kappab}_{\RP,w_n}(t_0^n,\vartheta(t_0^n)) \right)}(2 + \overline{\Theta}) n^{-\beta}$. But now, we need to consider the times $\vartheta\circ\vartheta(t_0^n) \equiv \vartheta^2(t_0^n)$, $\vartheta^3(t_0^n), \dots, \vartheta^J(t_0^n)$, as long as $\vartheta^J(t_0^n) < T\wedge t_{1}^n$. By an immediate induction, we obtain that for such $J$,
		\begin{align*}
		\|Y-\Pi_\cdot(Y^n)\|_{\infty,[t_0^n,\vartheta^J(t_0^n)]} \leq 2^J \exp\left\{J\vee \left(C \sum_{j=0}^{J-1} \overline{\kappab}_{\RP,w_n}(\vartheta^j{t_0^n},\vartheta^{j+1}(t_0^n)) \right)\right\} (2 + J\overline{\Theta}) n^{-\beta} ,
		\end{align*}
		where we used $\vartheta^0(t_0^n) = t_0^n$. Note however that $\vartheta^J(t_0^n) = J\delta_\RP$, so that $J$ must be smaller than $T\delta_\RP^{-1}$. Besides, by the super-additivity property of $\overline{\kappab}_{\RP,w_n}$, one gets
		\begin{align*}
		\|Y-\Pi_\cdot(Y^n)\|_{\infty,[t_0^n,\vartheta^J(t_0^n)]} \leq 2^{T\delta_\RP^{-1}} \exp\left\{T\delta_\RP^{-1}\vee \left(C \overline{\kappab}_{\RP,w_n}(t_0^n,T) \right)\right\} (2 + T\delta_\RP^{-1} \overline{\Theta}) n^{-\beta},
		\end{align*}
		which is smaller than the right-hand side of \eqref{eq:boundRate}.
		
		To conclude this step, observe that if $\vartheta^{J+1}(t_{0}^n)=T$, then we proved \eqref{eq:boundRate}. While if $\vartheta^{J+1}(t_{0}^n)= t_{1}^n$, then one can move to point $2.$ and iterate.
	\end{enumerate}
	
	Following this construction, there exists an increasing sequence $(t_k^n)_{k\in\N} \in [0,T]^\N$ (possibly taking only finitely many different values) such that $\lim_{k\rightarrow \infty} t_{k}^n=T$ and for any $k\geq 0$, $\|Y-\Pi_\cdot(Y^n)\|_{\infty,[t_{2k}^n,t_{2k+1}^n]}$ is bounded by the right-hand side of inequality \eqref{eq:boundRate}. This achieves the proof of this theorem.
\end{proof}

\begin{proof}[Proof of Theorem~\ref{th:rateGaussian}]
	First, we provide a bound on $K_T$ where $(Y,K)$ is the solution of $SP(\sigma,\RP,L)$. Similarly, observe that $(L,0)$ is solution of $SP(0,0,L)$. We call Skorokhod mapping the function that takes any continuous paths $(z,l)$ and maps it to $(y,k)$, where $y=z+k$ is a path reflected on $l$. The Skorokhod mapping is Lipschitz continuous in the uniform topology (call $C_S$ the Lipschitz constant), see for instance Equations (2.1)-(2.2) in \cite{Slominski}. Thus one gets that
	\begin{align}\label{eq:boundKT}
	K_T = \|K\|_{\infty,[0,T]} &\leq C_S \|y_0 + \int_0^\cdot \sigma(Y_u) \dd\RP_u - L\|_{\infty,[0,T]} \nonumber\\
	&\leq C_S \left( C(1+\Theta) (\|X\|_\beta+\|\mathbb{X}\|_{2\beta})  + \|L\|_\beta \right)T^\beta
	\end{align}
	where we used Corollary~\ref{cor:unifHolder}. 
	But $\Theta$ depends linearly on $C^\RP_J$, and as already mentioned, this quantity may not have exponential moments. Since an exponential of $\Theta$ appears in Inequality \eqref{eq:finalBound}, this explains why we cannot get a simple upper bound for $\EE\left( \|\Delta_n Y\|_{\infty,[0,T]} \right)$. From Lemma~\ref{lem:unifBoundRYn} and the definition of $\overline{\Theta}$ in Lemma~\ref{lem:bound1sec5}, recall that for any $ \gamma\geq 1$, $\EE(\overline{\Theta}^\gamma) <\infty$. Moreover, we know from \eqref{eq:boundDeltaX} that $\EE(\delta_\RP^\gamma) <\infty$. Finally, since $K^n_T \leq K_T$, one gets from \eqref{eq:defKappaXwn} and \eqref{eq:defKappaBar} that 
	\begin{align*}
	\overline{\kappab}_{\RP,w_n}(0,T) \leq C\left(1+\|\mathbb{X}\|_{\ppv,[0,T]}^p \right) \left( \sum_{k=1}^4 
	\left(
	\vertiii{\RP}_{p,[0,T]}^{kp} +K_T^{kp}
	\right)
	\vee \sum_{k=1}^4 \left(\vertiii{\RP}_{p,[0,T]} ^{kp^2} +K_T^{kp^2}\right)\right) .
	\end{align*}
	Hence, in view of the bound \eqref{eq:boundKT} on $K_T$, there exists a random variable $\overline{\kappab}$ such that, for any $\gamma\geq 1$,
	\begin{align*}
	\EE\left( (\sup_{n\in\N}\overline{\kappab}_{\RP,w_n}(0,T))^\gamma \right) = \EE\left(\overline{\kappab}^\gamma \right) <\infty .
	\end{align*}
	Using \eqref{eq:finalBound}, it is now clear that for any $\gamma\geq 1$, 
		$\EE\left[  \log \left(1+\sup_{n\in\N} \big(n^\beta \|Y-Y^n\|_{\infty,[0,T]} \big)\right)^\gamma \right]  <\infty$.
\end{proof}

\section{Application: existence of a density for the reflected process}\label{sec:density}

In this last section, we aim at proving Theorem~\ref{th:density}.  Recall that $(\Omega,\mathcal{F},\PP)$ is a complete probability space. Thus let us consider the following simplified problem (compared to \eqref{eq:reflected_fSDE}), with constant diffusion coefficient and one-dimensional fractional Brownian noise:
\begin{equation*}
\forall t\geq 0,\quad Y_t = y_0 + \int_0^t b(Y_u) \dd u + K_t + B^H_t,
\end{equation*}
where $B^H$ is a fractional Brownian motion with Hurst parameter $H\in [\tfrac{1}{2},1)$, $y_0\geq 0$, $b\in \mathcal{C}^1_b(\R)$ and $(Y,K)$ is the solution of the Skorokhod problem above the constant boundary $0$. Note that we added a drift $b$ compared to previous equations, which in the previous section could have been part of the vector field $\sigma$. However it is not necessary here to assume as much as a fourth order bounded derivative in this case, and one can check that $b\in \mathcal{C}^1_b$ is enough to get existence with the method of Section~\ref{sec:penalisation}.\\
As in the previous sections, $Y$ is approximated by the non-decreasing sequence of processes $(Y^n)_{n\in\N}$:
\begin{equation}\label{eq:YnSig1}
\forall t\geq 0,\quad Y^n_t = y_0 + \int_0^t b(Y^n_u) \dd u + \int_0^t \psi_n(Y^n_u) \dd u + B^H_t.
\end{equation}

\subsection{Malliavin calculus and fractional Brownian motion}

Let us briefly review some fundamental tools and results of Malliavin calculus that permit to prove that some random variables are absolutely continuous with respect to the Lebesgue measure. In a second paragraph, we shall give a brief account of Malliavin calculus for the fractional Brownian motion, in a manner that emphasises the applicability of the general results to the fBm framework.

~

Let $D$ denote the usual Malliavin derivative on the Cameron-Martin space $\HH=L^2([0,T])$. For any $p\geq 1$, let $\mathbb{D}^{1,p}$ be the Malliavin-Sobolev space associated to the derivative operator $D$. 
We aim at applying the following result of Bouleau and Hirsch \cite{BouleauHirsch} to $Y_t$.
\begin{theorem}[Th. 2.1.3 of \cite{Nualart}]\label{th:existDensityBH}
	Let $X\in \mathbb{D}^{1,2}$ be a real-valued random variable. If $A\in \mathcal{F}$ 
	and $\|DX\|_{\mathcal{H}}>0~a.s.$ on $A$, then the restriction of the law of $X$ to $A$, i.e. the measure $[\mathbf{1}_A \PP]\circ X^{-1}$, admits a density with respect to the Lebesgue measure.
\end{theorem}

~

Concerning the Malliavin calculus for fBm, we recall some definition and properties borrowed from \cite[Chapter 5]{Nualart} (see also \cite[Section 2]{RichardTalay} for a more detailed introduction than what we present here). One possible approach to this calculus is to consider the kernel $\Gamma$, defined for any $H\in(\tfrac{1}{2},1)$ by:
\begin{equation}
\Gamma(t,s) = 
\begin{cases}
c_H s^{\frac{1}{2}-H}\int_s^t u^{H-\frac{1}{2}} (u-s)^{H-\frac{3}{2}} \dd u  & \text{ if }t>s>0 \\
0 & \text{ if } t\leq s,
\end{cases}
\end{equation}
where $c_H$ is a positive constant (see \cite[Eq. (5.8)]{Nualart}). Note that for the standard Brownian motion ($H=\tfrac{1}{2}$), $\Gamma(t,s) = \mathbf{1}_{[0,t]}(s)$. Then if $B^H$ is a fractional Brownian motion,
\begin{align*}
D_s B^H_t = \Gamma(t,s) .
\end{align*}
Furthermore, one can define an $L^2([0,T])$-valued linear isometry $\Gamma^*$ as follows: for any step function $\varphi$,
\begin{align*}
\Gamma^*\varphi(s) = \int_s^T \varphi(u) \frac{\partial \Gamma}{\partial u}(u,s) \dd u = c_H \int_s^T \varphi(u) \left(\frac{u}{s}\right)^{H-\frac{1}{2}} (u-s)^{H-\frac{3}{2}} \dd u .
\end{align*}
The domain of $\Gamma^*$ is thus a Hilbert space, which we do not need to characterise here, but only recall that for any $H\in[\tfrac{1}{2},1)$, it contains $\mathcal{C}_b([0,T])$. Besides, if the support of $\varphi$ is contained in $[0,t]$ for some $t>0$, one can verify that
\begin{align}\label{eq:Kstar}
\Gamma^*\varphi(s) = \int_s^t \frac{\partial \Gamma}{\partial u}(u,s) \varphi(u)  \dd u .
\end{align}
Applying the Malliavin derivative on both sides of \eqref{eq:YnSig1}, we obtain
\begin{equation*}
\forall s,t\geq 0,\quad D_s Y^n_t = \int_0^t D_s Y^n_u ~b'(Y^n_u) \dd u + \int_0^t D_s Y^n_u ~\psi_n'(Y^n_u) \dd u + \Gamma(t,s) .
\end{equation*}
Since for any fixed $s\geq 0$, the previous equality is a linear ODE in $t$, solving it yields
\begin{equation}\label{eq:derivYn}
\forall t\geq 0,\quad D_s Y^n_t = \Gamma^*\left[\mathbf{1}_{[0,t]}(\cdot) \exp\left\{\int_{\cdot}^t (b'(Y^n_v) + \psi_n'(Y^n_v)) \dd v \right\} \right](s)  .
\end{equation}
For the ease of notations, let us define, for $\varphi= b'+\psi_n'$ or $\varphi=b'$, the process
\begin{align*}
\mathcal{E}_{s,t}[\varphi] = \mathbf{1}_{[0,t]}(s) \exp\left\{\int_{s}^t \varphi(Y^n_v) \dd v \right\} .
\end{align*}

\subsection{Proof of existence of a density (Theorem \ref{th:density})}

Let $t>0$. The scheme of proof is very similar to the one used by \citet{Tindel} for elliptic PDEs perturbed by an additive white noise. First, notice that  the sequence $(DY^n_t)_{n\in\N}$ is bounded in $L^2(\Omega;\HH)$. Indeed, the mapping $\mathcal{E}_{\cdot,t}[b'+\psi_n']$ is bounded uniformly in $n$ and $s\in[0,T]$, since $\psi_n'\leq 0$. Hence it is clear that the same is true of $\Gamma^*Y^n_t$ (note that this is where things become difficult if one wants to consider the case $H<\tfrac{1}{2}$). 
Since $Y^n$ converges to $Y$ in $L^2$ (cf Theorem~\ref{th:RfSDE_Gaussian}) and $\sup_{n\in\N}\EE\left(\|DY^n_t\|_\HH\right)<\infty$ in view of the preceding discussion, we deduce that $Y_t\in \mathbb{D}^{1,2}$ and that $(DY^n_t)_{n\in\N}$ converges weakly to $DY_t$ in $L^2(\Omega;\HH)$ (see for instance Lemma 1.2.3 in \cite{Nualart}), i.e.  that for any $\chi \in L^2(\Omega)$ and any $f \in \HH$,
\begin{align}\label{eq:limitDYn}
\lim_{n\rightarrow \infty} \EE\left[ \chi \langle DY^n_t, f\rangle \right] = \EE\left[ \chi \langle DY_t, f\rangle \right] .
\end{align}
In the sequel, we shall apply this convergence to any non-negative $\chi\in L^2(\Omega)$, and to any $f\in \HH$ with a sufficiently small support.

Let $\Omega_0$ be a measurable set of measure $1$ on which $Y^n$ converges (uniformly) (cf Theorem~\ref{th:RfSDE_Gaussian}). In view of Theorem~\ref{th:existDensityBH}, it suffices to prove that for any $a>0$, $\|DY_t\|_\HH>0~ a.s.$ on the event 
\begin{align*}
\Omega_a = \{\omega\in \Omega_0:~ Y_t\geq 3a\} .
\end{align*}
As in \cite{Tindel}, we notice that it suffices to prove that for any $k,j\in\N^*$, $\|DY_t\|_\HH>0~ a.s.$ on the following event
\begin{align*}
\Omega_{a,k,j} := \left\{\omega\in \Omega_a:~ Y^k_t\geq 2a \text{ and } Y^k_s\geq a,~\text{for any $s$ such that } |t-s|\leq j^{-1} \right\} .
\end{align*}
Indeed  $\bigcup_{k,j\in \N^*} \Omega_{a,k,j} = \Omega_a$ (since by definition $\Omega_{a,k,j} \subseteq \Omega_{a}$).

Hence let now $a,k,j$ be fixed. Since the sequence $Y^n$ is non-decreasing, $Y^n_s \geq a~a.s.$ on $\Omega_{a,k,j}$ for all $n\geq k$ and all $s\in [t-j^{-1},t+j^{-1}]$, and thus $\psi_n'(Y^n_s)=0$. Hence on $\Omega_{a,k,j}$ and for all $n\geq k$, \eqref{eq:derivYn} now reads
\begin{align}\label{eq:DYn=KE}
\forall s\in [t-j^{-1},t],\quad  D_s Y^n_t = \Gamma^*\mathcal{E}_{\cdot,t}[b'](s) .
\end{align}
Based on the definitions of the previous section (in particular \eqref{eq:Kstar}), and for any non-negative $f\in L^2$ with support in $[t-j^{-1},t]$,
\begin{align}\label{eq:scalprod}
\langle DY^n_t, f \rangle &= c_H \int_{t-j^{-1}}^t f(s) \int_s^t \left(\frac{u}{s} \right)^{H-\frac{1}{2}} (u-s)^{H-\frac{3}{2}} \mathcal{E}_{u,t}[b'+\psi_n'] \dd u  \dd s.
\end{align}
Thus using \eqref{eq:DYn=KE} and the boundedness of $b'$, it follows that $\mathbf{1}_{\Omega_{a,k,j}} \mathcal{E}_{u,t}[b'+\psi_n'] \geq \mathbf{1}_{\Omega_{a,k,j}} e^{-(t-u)\|b'\|_\infty}$, hence one gets
\begin{align*}
\mathbf{1}_{\Omega_{a,k,j}}\langle DY^n_t, f \rangle &\geq \mathbf{1}_{\Omega_{a,k,j}} e^{-j^{-1}\|b'\|_\infty}  \int_{t-j^{-1}}^t  f(s) c_H \int_s^t \left(\frac{u}{s} \right)^{H-\frac{1}{2}} (u-s)^{H-\frac{3}{2}}  \dd u \dd s \\
&= \mathbf{1}_{\Omega_{a,k,j}} e^{-j^{-1}\|b'\|_\infty}  \int_{t-j^{-1}}^t  f(s) \Gamma(t,s) \dd s .
\end{align*}
Hence the previous inequality and \eqref{eq:limitDYn} yield that for any non-negative $\chi\in L^2(\Omega)$ and any non-negative $f\in L^2$ with support in $[t-j^{-1},t]$,
\begin{align*}
\EE\left( \mathbf{1}_{\Omega_{a,k,j}} \chi \langle DY_t, f \rangle \right) \geq \EE\left(\mathbf{1}_{\Omega_{a,k,j}} \chi~e^{-j^{-1}\|b'\|_\infty} \langle \Gamma(t,\cdot), f\rangle \right). 
\end{align*}
It follows that almost surely, $\mathbf{1}_{\Omega_{a,k,j}} D_sY_t \geq \mathbf{1}_{\Omega_{a,k,j}} e^{-j^{-1}\|b'\|_\infty} \Gamma(t,s)$ for almost all $s\in [t-j^{-1},t]$. Hence the following inequality holds almost surely on $\Omega_{a,k,j}$:
\begin{align*}
\|DY_t\|_\HH \geq e^{-j^{-1}\|b'\|_\infty} \|\Gamma(t,\cdot)\mathbf{1}_{[t-j^{-1},t]}\|_\HH  >0 .
\end{align*}
Applying Theorem~\ref{th:existDensityBH} to $DY_t$ concludes the proof of Theorem~\ref{th:density}.

\appendix

\section{Proofs and \emph{a priori} estimates}\label{app}

\subsection{Proof of technical results}\label{subsec:app1}

\begin{proof}[Proof of Proposition~\ref{prop:ExistenceUnbdd}]
	 In a first step, we adopt the definition of solution given in \cite[Definition 12.1]{FrizVictoir} (see also \cite[Definition 3]{FrizOberhauser}),  which gives $Y$ as the uniform limit of some paths $Y^k$ which solve \eqref{eq:RDEdrift} with $\RP$ replaced by an approximating sequence of Lipschitz paths $(X^k, \mathbb{X}^k)$.
 In the last paragraph of this proof, we explain why this solution is also a controlled solution.

	The local existence of a solution of \eqref{eq:penalized_smoothed_fSDE} comes from \cite[Theorem 3]{LejayRDE} (see also \cite[Theorem 10.21]{FrizVictoir}). Uniqueness is also granted given the regularity of $\sigma$ and $b$. In view of \cite[Lemma 1]{LejayRDE} (see also \cite[Theorem 10.21]{FrizVictoir}), we know that either $Y$ is a global solution on $[0,T]$, or that there is a time $\tau'\leq T$ such that for any $\tau\in[0,\tau')$, $(Y_s)_{s\in[0,\tau]}$ is a solution to \eqref{eq:penalized_smoothed_fSDE} and that $\lim_{t\rightarrow \tau'} |Y_t| = \infty$. Thus we shall prove that $Y_t$ coincides on $[0,\tau')$ with the solution to \eqref{eq:def_XY}. Since the latter does not explode, it will follow that $Y$ is a global solution.

	Let us turn to the Doss-Sussmann representation. Recall from Subsection~\ref{subsec:flow} that if $\sigma\in \mathcal{C}_b^4$, $\Jac[\RP][\cdot][0][t]$ is Lipschitz continuous, uniformly in $t$, 
	but that due to the unboundedness of $b$, $W(t,\cdot)$ is only locally Lipschitz (uniformly in $t$). This suffices to prove existence and uniqueness of a solution to $\dot{z}_t = W(t,z_t)$ on a small enough time interval. In fact, $\Jac[\RP][\cdot][0][t]$ is bounded (see \eqref{eq:boundUJJinv}) and $C^\RP_U:= \sup_{t\in[0,T]} |\flow[\RP][0][t][0]|<\infty$. Denote by $B(C^\RP_U)$ the ball of $\R^e$ centred in $0$ and with radius $C^\RP_U$. Thus 
	\begin{align*}
	|W(t,z)| &\leq |\Jac[\RP][z][0][t] b(\flow[\RP][0][t][0])| + |W(t,z)-\Jac[\RP][z][0][t] b(\flow[\RP][0][t][0])| \\
	&\leq C^{\RP}_J \sup_{x\in B(C^\RP_U)}|b(x)| + C^\RP_J |b(\flow[\RP][z][t][0])-b(\flow[\RP][0][t][0])|\\
	&\leq C^{\RP}_J \left( \|b\|_{\infty, B(C^\RP_U)} + \|\nabla b\|_\infty C^{\RP}_J |z|\right),
	\end{align*}
	i.e. $W$ has linear growth. This ensures the stability of the solution $Z_t$ to the ODE $\dot{Z}_t = W(t,Z_t)$, and its global existence on any time interval (see e.g. \cite[Theorem  3.7]{FrizVictoir}). Thus the process $(\flow[\RP][Z_t])_{t\in[0,T]}$ is well-defined.
	
	Now to prove that $(Y_t)_{t\in[0,\tau]}$ and $(\flow[\RP][Z_t])_{t\in[0,\tau]}$ coincide, one can follow the scheme of proof of \cite[Proposition 3]{FrizOberhauser}: let $(\RP^k)_{k\in\N}$ a sequence of geometric rough paths such that $(X^k)_{k\in\N}$ is a sequence of Lipschitz paths with uniform $\beta$-Hölder bound, which converges pointwise to $\RP$. Denote by $Y^{k}$ the solution to \eqref{eq:penalized_smoothed_fSDE} where $\RP^k$ replaces $\RP$. Then $(Y^{k},Z^{k})$ is easily seen to solve \eqref{eq:def_XY} with $\RP$ replaced by $\RP^k$. As stated in \cite{FrizOberhauser}, it suffices to prove the uniform convergence of $Z^{k}$ to $Z$ to get the result. Define $$M^k = \displaystyle\sup_{s\in[0,\tau], z\in\R^e} |\Jac[\RP][z][0][s] - \Jac[\RP^k][z][0][s]|\vee |\flow[\RP][z][s] - \flow[\RP^k][z][s]|$$ and denote $J_{\text{Lip}} = \sup_{s\in[0,T]} \|\Jac[\RP][\cdot][s]\|_{\text{Lip}}$ which is finite (see the discussion of Section~\ref{subsec:flow}). Now the main difference with \cite{FrizOberhauser} lies again in the unboundedness of $b$: denote by \(\overline{Z} = \sup_{t\in[0,T]} |Z_t| <\infty\) and  $\overline{C}^\RP_U = \sup_{t\in[0,T], z\in B(0,\overline{Z})} |\flow[\RP][z][t]|<\infty$. 
	Then for $t\leq \tau$, 
	\begin{align*}
	|Z^{k}_t - Z_t| &\leq \int_0^t \bigg\{ |\Jac[\RP][Z_s][0][s]-\Jac[\RP][Z^{k}_s][0][s]|~|b(\flow[\RP][Z_s][s])| +|\Jac[\RP][Z^{k}_s][0][s]-\Jac[\RP^k][Z^{k}_s][0][s]|~|b(\flow[\RP^k][Z^{k}_s][s])|  \\
	&\hspace{1cm}  + |\Jac[\RP][Z^{k}_s][0][s]| | b(\flow[\RP^k][Z^{k}_s][s][0]) - b(\flow[\RP][Z_s][s][0])|\bigg\} \dd s\\
	&\leq \int_0^t\bigg\{ J_{\text{Lip}} |Z^{k}_s - Z_s|~\|b\|_{\infty, B(C^\RP_U)}  + M^k ~ \left(\|b\|_{\infty,B(\overline{C}^\RP_U)} + \|\nabla b\|_\infty |\flow[\RP^k][Z^{k}_s][s]- \flow[\RP][Z_s][s]|  \right) \\
	&\hspace{1cm}  +C^\RP_J \|\nabla b\|_{\infty} |\flow[\RP^k][Z^{k}_s][s]- \flow[\RP][Z_s][s]| \bigg\} \dd s \\
	&\leq \int_0^t\bigg\{ J_{\text{Lip}} |Z^{k}_s - Z_s|~\overline{b}  + M^k ~ \overline{b} + \|\nabla b\|_{\infty} \left(C^\RP_J+M^k\right)\left(M^k + C^\RP_J |Z^k_s-Z_s|\right)\bigg\} \dd s,
	\end{align*}
	where $\overline{b}:= \max\left(\|b\|_{\infty, B(C^\RP_U)}, \|b\|_{\infty, B(\overline{C}^\RP_U)}\right)$. Then, denoting 
	$C^{Z,1} := \overline{b} + \|\nabla b\|_{\infty} C^\RP_J$ and $C^{Z,2} := \overline{b}J_{\text{Lip}} + \|\nabla b\|_\infty (C^\RP_J)^2$, one gets
	\begin{align}\label{eq:boundZk-Z}
	\sup_{s\in[0,t]} |Z^{k}_s - Z_s| &\leq t C^{Z,1} M^k+ t \|\nabla b\|_\infty (M^k)^2 + C^{Z,2}(1+M^k) \int_0^t \sup_{u\in[0,s]} |Z^{k}_u - Z_u| \dd s \nonumber\\
	&\leq  \left(C^{Z,1} + \|\nabla b\|_\infty M^k \right) T M^k~\exp\left(C^{Z,2}(1+M^k) T \right),
	\end{align}
	by applying Gr\"onwall's lemma in the last inequality. By the continuity of the mapping $(y,\RP) \mapsto \flow[\RP][y][\cdot]$ (see \cite[Theorem 8.5]{FrizHairer}), there is $M^k \rightarrow 0$ as $k\rightarrow \infty$. Hence the inequality \eqref{eq:boundZk-Z} implies that $Z^{k}$ converges uniformly to $Z$. Since $Y^{k}$ has the representation \eqref{eq:def_XY}, then so has $Y$.
	
	Hence $(Y_t)_{t\in[0,\tau]}$ and $(\flow[\RP][Z_t])_{t\in[0,\tau]}$ do coincide and since the latter does not explode in finite time, this implies that there cannot exist $\tau'>\tau$ such that $\lim_{t\rightarrow \tau'} |Y_t| = \infty$. Thus $Y$ is defined on $[0,T]$.

	Finally, denoting \(M = \sup_{t\in[0,T]}|Y_t|\) and up to changing the drift $b$ into a bounded smooth function \(b^{(M)}\) equal to \(b\) on the interval \((-2M,2M)\), we get that $Y$ solves an RDE with bounded smooth coefficients in the sense of controlled rough paths (see \cite[Theorem 8.4]{FrizHairer}).
\end{proof}

\begin{proof}[Proof of Lemma~\ref{lem:deterministic}]
	\begin{enumerate}[label = (\roman*)]
		\item Denoting $k^n_t:=\Psi\int_0^t \psi_n(f^n_u-\ell_u) \dd u$, let $\overline{f}^n$ and $\overline{g}^n$ be defined as follows:
		\begin{align*}
		\overline{f}^n_t := \delta f^n_{0,t} -\delta\ell_{0,t} &= -\delta\ell_{0,t} + g^n_t + \Psi\int_0^t \psi_n(f^n_u-\ell_u) \dd u\\
		&=: \overline{g}^n_t + k^n_t .
		\end{align*} 
		Observe that
		\begin{align*}
		(\overline{f}^n_t)^2 = (\overline{g}^n_t)^2 + (k^n_t)^2 + 2\int_0^t \overline{g}^n_t \dd k^n_u &= (\overline{g}^n_t)^2 + 2\int_0^t\left( k^n_u +\overline{g}^n_t\right) \dd k^n_u\\
		&\leq (\overline{g}^n_t)^2 + 2\int_0^t\left(\overline{g}^n_t - \overline{g}^n_u\right) \dd k^n_u
		\end{align*}
		where we used the inequality $\overline{f}^n_u \psi_n(f^n_u-\ell_u) \leq (f^n_u-\ell_u) \psi_n(f^n_u-\ell_u) \leq 0$. It follows that
		\begin{align*}
		(\overline{f}^n_t)^2\leq (\overline{g}^n_t)^2 + 2k^n_t \|\overline{g}^n_t - \overline{g}^n_\cdot\|_{\infty,[0,t]}&\leq (\overline{g}^n_t)^2 + 2(|\overline{f}^n_t| + |\overline{g}^n_t|) \|\overline{g}^n_t - \overline{g}^n_\cdot\|_{\infty,[0,t]}\\
		&\leq 5\|\overline{g}^n_\cdot\|_{\infty,[0,t]}^2 + 4|\overline{f}^n_t|\|\overline{g}^n_\cdot\|_{\infty,[0,t]}\\
		&\leq 5\|\overline{g}^n_\cdot\|_{\infty,[0,t]}^2 + \tfrac{1}{2}\left(|\overline{f}^n_t|^2 + 16\|\overline{g}^n_\cdot\|_{\infty,[0,t]}^2\right),
		\end{align*}
		which implies the result.

		\item 
		The inequality $\psi_n(x) \geq -\tfrac{1}{2} - nx$ yields $f^n_t -\ell_t\geq f^n_0 - \ell_0 +\overline{g}^n_t-\tfrac{1}{2}\Psi t - n\Psi\int_0^t (f^n_u-\ell_u) \dd u$. Denote $\widetilde{g}^n_t:= \overline{g}^n_t-\tfrac{1}{2}\Psi t$ and $\widetilde{f}^n$ the solution to
		\begin{align}\label{eq:tildef}
		\widetilde{f}^n_t -\ell_t&= f^n_0 - \ell_0 +\widetilde{g}^n_t - n\Psi\int_0^t (\widetilde{f}^n_u-\ell_u) \dd u.
		\end{align}
		The comparison principle of ODEs implies that $\forall t\in[0,T]$, $f^n_t-\ell_t\geq \widetilde{f}^n_t-\ell_t$. Solving \eqref{eq:tildef} yields
		\begin{align}\label{eq:lowBoundfn-l}
		f^n_t-\ell_t&\geq (f^n_0 - \ell_0) e^{-n\Psi t} -  \int_0^t e^{-n\Psi(t-u)}\dd \widetilde{g}^n_u \nonumber\\
		&\geq (f^n_0 - \ell_0) e^{-n\Psi t} +\widetilde{g}^n_t e^{-n\Psi t}+ n\Psi \int_0^t e^{-n\Psi(t-u)} (\widetilde{g}^n_u-\widetilde{g}^n_t)\dd u~,~~t\in[0,T].
		\end{align} 
		Since $\psi_n(x) \leq n x_-$, we now obtain from \eqref{eq:lowBoundfn-l} that
		\begin{align}\label{eq:boundPsin}
		\psi_n(f^n_t-\ell_t) &\leq n\left(\widetilde{g}^n_t e^{-n\Psi t} + n\Psi \int_0^t e^{-n\Psi(t-u)} (\widetilde{g}^n_u-\widetilde{g}^n_t)\dd u\right)_- \nonumber\\
		&\leq n \|\widetilde{g}^n\|_{\beta} t^\beta e^{-n\Psi t}  + \|\widetilde{g}^n\|_\beta n^2 \Psi \int_0^t e^{-n\Psi(t-u)} (t-u)^\beta \dd u~,~~t\in[0,T].
		\end{align}
		It is clear that $n \|\widetilde{g}^n\|_{\beta} t^\beta e^{-n\Psi t} \leq \|\widetilde{g}^n\|_{\beta} \Psi^{-\beta} n^{1-\beta}$. Thus one focuses now on the second term: an integration by parts and the change of variables $v=n\Psi u$ yield
		\begin{align*}
		n^2 \Psi \int_0^t e^{-n\Psi(t-u)} (t-u)^\beta \dd u &= -n t^\beta e^{-n\Psi t} + \beta n \int_0^t e^{-n\Psi u} u^{\beta-1}\dd u\\
		& =  -n t^\beta e^{-n\Psi t} + \beta n^{1-\beta} \Psi^{1-\beta} \int_0^{n\Psi t} v^{\beta-1} e^{-v}\dd v \\
		&\leq C n^{1-\beta} \Psi^{1-\beta}.
		\end{align*}
		Plugging the last inequality in \eqref{eq:boundPsin} gives the desired result.
	\end{enumerate}
\end{proof}


\subsection{A priori estimates for controlled RDEs with finite variation drift}\label{subsec:app2}

In the whole subsection, we consider two non-decreasing continuous paths $K^1$ and $K^2$, and for some fixed 
$\sigma\in \mathcal{C}^3_b(\R, (\R^d)')$, we may assume that there exist solutions $Y^1$ and $Y^2$ on $[0,T]$ to the following equations
\begin{equation*}
\begin{cases}
dY^i_t = dK^i_t + \sigma(Y^i_t) d\RP_t \\
Y^i_0 = y^i \in \R
\end{cases}
,~ t\in[0,T],~i=1,2 ,
\end{equation*}
in the sense of controlled rough paths. Let us set for any $(s,t)\in \mathcal{S}_{[0,T]}$,
\begin{equation}\label{eq:defRRs}
\begin{split}
&R^{Y^i}_{s,t} := \delta Y^i_{s,t} -\sigma(Y^i_s) \delta X_{s,t} , \\
&R^{\sigma(Y^i)}_{s,t} := \delta\sigma(Y^i)_{s,t} -\sigma'(Y^i_s) \sigma(Y^i_s) \delta X_{s,t} .
\end{split}
\end{equation}
Recall that the operator $\Delta$ was introduced in Subsection~\ref{subsec:uniqueness}, and similarly to \eqref{eq:defKappaXwn}, consider $\kappab_{\RP,K^i}(s,t) := \vertiii{\RP}_{p,[s,t]}^p +  (\delta K^i_{s,t})^{p}$.

\begin{lemma}\label{lem:RDeltaSig}
	Let $K^1$ and $K^2$ be non-decreasing continuous paths, let $\sigma\in \mathcal{C}^3_b(\R, (\R^d)')$, let $\RP \in \Cbeta$ with 
	$\beta\in(\tfrac{1}{3},\tfrac{1}{2})$, and assume that $Y^1$ and $Y^2$ are controlled solutions as above. Then there exists $C>0$ 
	depending only on $p(=\beta^{-1})$ and the uniform norm of $\sigma$ and its derivatives, and there exists 
	$\delta_\RP \geq C^{-1} \vertiii{\RP}_\beta^{-1/\beta} \wedge T$ such that: 
\begin{enumerate}[label = (\roman*)]
\item\label{item:i} For any $(s,t)\in \mathcal{S}_{[0,T]}$,
\begin{align*}
	\|Y^i\|_{p,[s,t]}^p \leq C ~\phi_p\left(\kappab_{\RP,K^i}(s,t) \right) ,~i=1,2.
\end{align*}

\item\label{item:i'} For any $(s,t)\in \mathcal{S}_{[0,T]}$ such that $|t-s|\leq \delta_\RP$,
\begin{equation}\label{eq:boundRRspv}
	\max\left(\|\Ri\|_{\ppv,[s,t]}^{\frac{p}{2}}, \|\Rsi\|_{\ppv,[s,t]}^{\frac{p}{2}} \right) \leq C~ \kappab_{\RP,K^i}(s,t), ~i=1,2.
\end{equation}

\item\label{item:ii} For any $(s,t)\in \mathcal{S}_{[0,T]}$ such that $|t-s|\leq \delta_\RP$,
	\begin{align*}
	\|R^{\Delta\sigma(Y)}\|_{\ppv,[s,t]} \leq C \left\{w_K(s,t) + (w_K(s,t)+\|\Delta Y\|_{\infty,[s,t]})
	\phi_p\left(\widetilde{\kappab}_{\RP,K}(s,t)^{\frac{1}{p}} \right) \right\} .
	\end{align*}
\end{enumerate}
\end{lemma}

\begin{remark}\label{rem:pvar}
Assertions \ref{item:i} and \ref{item:i'} are not surprising in light of classical \emph{a priori} estimates (cf \cite[Proposition 8.3]{FrizHairer}), but the novelty is that the control depends on $K$ and that we use the identity 
\[
\delta K^i_{s,t}  = \|K^i\|_{p,[s,t]}.
\]
This equality, which relies on the fact that $K^i$ is non-decreasing, is crucial in Section~\ref{sec:penalisation} to establish the existence of $Y$ and the uniform \emph{a priori} estimates of Lemma~\ref{lem:unifBoundRYn}.
\end{remark}

\begin{proof}
	\textbf{$1^{\text{st}}$ step.} By a Taylor expansion,
	\begin{align*}
	\delta \sigma(Y^i)_{s,t} = \sigma'(Y^i_s) \delta Y^i_{s,t} + \int_{Y^i_s}^{Y^i_t} \sigma''(y) (Y^i_t-y) \dd y .
	\end{align*}
	Hence the combination of \eqref{eq:defRRs} and the previous equality yields
	\begin{align}\label{eq:RsigRi}
	\Rsi_{s,t} = \sigma'(Y^i_s) \Ri_{s,t} + \int_{Y^i_s}^{Y^i_t} \sigma''(y) (Y^i_t-y) \dd y 
	\end{align}
	and
	\begin{align}\label{eq:bound1Rsi}
	|\Rsi_{s,t}|\leq \|\sigma'\|_{\infty} |\Ri_{s,t}| + \tfrac{1}{2}\|\sigma''\|_{\infty} (Y^i_t-Y^i_s)^2.
	\end{align}

	Since we assumed that $Y^1$ and $Y^2$ are controlled by $X$, the definition \eqref{eq:defRRs} of $\Ri$  and inequality \eqref{eq:boundRI} applied to $|\int_s^t \sigma(Y^i_u)\dd \RP_u - \sigma(Y^i_s)\delta X_{s,t}|$ yield
	\begin{align*}
	|\Ri_{s,t}| \leq \delta K^i_{s,t} + |\sigma'(Y^i_s)\sigma(Y^i_s)| |\mathbb{X}_{s,t}| + C_p\left( \|X\|_{\pv,[s,t]} \|\Rsi\|_{\frac{p}{2},[s,t]} + \|\sigma'(Y^i)\sigma(Y^i)\|_{\pv,[s,t]} \|\mathbb{X}\|_{\frac{p}{2},[s,t]} \right) .
	\end{align*}
	Setting $M = C_p (1+\|\sigma'\|_{\infty} + \tfrac{1}{2}\|\sigma''\|_{\infty})$ and
\begin{align*}
	\delta_\RP := T\wedge \sup\left\{\delta>0:~  \kappab_{\RP}(s,t)^{\frac{1}{p}}\leq \tfrac{1}{2} M^{-1} ,~\forall s,t\in[0,T] ~\text{s.t. } |t-s|\leq \delta\right\},
	\end{align*}
using \eqref{eq:bound1Rsi} and a standard argument \cite[p.110]{FrizHairer}, one obtains 
that for any $(s,t)\in\mathcal{S}_{[0,T]}$ such that $|t-s|\leq \delta_\RP$,
	\begin{align}\label{eq:bound1Ri}
	\|\Ri\|_{\ppv,[s,t]}
	&\leq 2\delta K^i_{s,t} + (2M+1)\|\mathbb{X}\|_{\ppv,[s,t]} + 2\|Y^i\|_{\pv,[s,t]}^2 .
	\end{align}
	Since we assumed that $\vertiii{\RP}_{\beta}<\infty$, observe that $\vertiii{\RP}_{\pv,[s,t]} \leq \vertiii{\RP}_{\beta} |t-s|^\beta$ and therefore we deduce that $\delta_\RP \geq (2 M \vertiii{\RP}_{\beta})^{-1/\beta}\wedge T$, as claimed in the statement of the Lemma.\\
	We recall briefly how to deduce a bound for $\|Y^i\|_p$. By the definition \eqref{eq:defRRs} of $\Ri$, and using estimate \eqref{eq:bound1Ri},
	\begin{align*}
	|\delta Y^i_{s,t}| 
	&\leq 2\delta K^i_{s,t} + (2M+1)\|\mathbb{X}\|_{\ppv,[s,t]} + 2\|Y^i\|_{\pv,[s,t]}^2 +\|\sigma\|_\infty \|X\|_{p,[s,t]}.
	\end{align*}
	Arguing as in \cite[p.111-112]{FrizHairer} with the only difference that there is now the term $\delta K^i_{s,t}$, one obtains (up to choosing a possibly larger $M$ which would still depend only on $p$ and the uniform norm of $\sigma$ and its derivatives) that for any $(s,t)\in\mathcal{S}_{[0,T]}$ such that $|t-s|\leq \delta_\RP$,
	\begin{align}\label{eq:boundYipv}
	\|Y^i\|_{p,[s,t]}^p 
	&\leq C ~\kappab_{\RP,K^i}(s,t), ~i=1,2,
	\end{align}
	for some $C>0$ that depends only on $p$ and $\sigma$. Besides, one can get a global upper bound as in \cite[Exercise 4.24]{FrizHairer}, obtaining that for any $(s,t)\in\mathcal{S}_{[0,T]}$,
	\begin{align*}
	\|Y^i\|_{p,[s,t]}^p &\leq C 
	\left(\kappab_{\RP,K^i}(s,t) \vee \kappab_{\RP,K^i}(s,t)^p\right)  
	= C ~\phi_p\left( \kappab_{\RP,K^i}(s,t)\right).
	\end{align*}
	The inequality \eqref{eq:boundRRspv} now follows easily from \eqref{eq:boundYipv}, \eqref{eq:bound1Rsi} and \eqref{eq:bound1Ri}. This achieves to prove \ref{item:i} and \ref{item:i'}.\\

	\textbf{$2^{\text{nd}}$ step.}  Given the definition of $R^{\Delta\sigma(Y)}$ and $R^{\Delta Y}$, we deduce from \eqref{eq:RsigRi} that
	\begin{align*}
	R^{\Delta\sigma(Y)}_{s,t} = & \sigma'(Y^1_s) R^{Y^1}_{s,t} + (\delta Y^1_{s,t})^2 \int_0^1 \sigma''(Y^1_t-(\delta Y^1_{s,t})u) u \dd u - \sigma'(Y^2_s) R^{Y^2}_{s,t} \\
	&\hspace{1cm}- (\delta Y^2_{s,t})^2 \int_0^1 \sigma''(Y^2_t-(\delta Y^2_{s,t})u) u \dd u\\
	=& \sigma'(Y^1_s) R^{\Delta Y}_{s,t} +\Delta\sigma'(Y)_s R^{Y^2}_{s,t} + \Delta (\delta Y_{s,t})^2 \int_0^1 \sigma''(Y^1_t-(\delta Y^1_{s,t})u) u \dd u \\
	&\hspace{1cm}+ (\delta Y^2_{s,t})^2 \int_0^1 \Delta \sigma''(Y_t-(\delta Y_{s,t})u) u \dd u.
	\end{align*}
	Hence
	\begin{align}\label{eq:boundRdeltasigYi}
	|R^{\Delta\sigma(Y)}_{s,t} | &\leq C \left\{ |R^{\Delta Y}_{s,t}| +|\Delta Y_s| |R^{Y^2}_{s,t}|+ |\Delta (\delta Y_{s,t})^2| + \|\Delta Y\|_{\infty,[s,t]} (\delta Y^2_{s,t})^2 \right\}.
	\end{align}
	We now provide an upper bound on $|R^{\Delta Y}_{s,t}|$ in terms of $R^{\Delta\sigma(Y)}_{s,t}$:
	\begin{align}\label{eq:boundRdeltaYi}
	|R^{\Delta Y}_{s,t}| &= |\delta(\Delta Y)_{s,t} - \Delta\sigma(Y)_s \delta X_{s,t}| \nonumber\\
	&= |\delta (\Delta K)_{s,t} + \int_s^t \Delta\sigma(Y)_u \dd\RP_u - \Delta\sigma(Y)_s \delta X_{s,t} | \nonumber\\
	&\leq w_K(s,t) + |\Delta\sigma'\sigma(Y)_s \mathbb{X}_{s,t}| + C_p\left(\|R^{\Delta \sigma(Y)}\|_{\ppv,[s,t]} \|X\|_{\pv,[s,t]} + \|\Delta\sigma'\sigma(Y)\|_{\pv,[s,t]} \|\mathbb{X}\|_{\ppv,[s,t]} \right).
	\end{align}
	Thus plugging \eqref{eq:boundRdeltaYi} into \eqref{eq:boundRdeltasigYi} yields
	\begin{align}\label{eq:boundRDeltaY}
	|R^{\Delta\sigma(Y)}_{s,t} | \leq C_p \|R^{\Delta \sigma(Y)}\|_{\ppv,[s,t]} \|X\|_{\pv,[s,t]} + C \bigg\{ & w_K(s,t) + |\Delta Y_s| (|R^{Y^2}_{s,t}| + |\mathbb{X}_{s,t}|) + |\Delta (\delta Y_{s,t})^2| \nonumber \\
	& + \|\Delta Y\|_{\infty,[s,t]} (\delta Y^2_{s,t})^2 +\|\Delta\sigma'\sigma(Y)\|_{\pv,[s,t]} \|\mathbb{X}\|_{\ppv,[s,t]}\bigg\}.
	\end{align}

	 \textbf{$3^{\text{rd}}$ step.} We now provide bounds for $|\Delta (\delta Y_{s,t})^2|$ and $\|\Delta\sigma'\sigma(Y)\|_{\pv,[s,t]}$.
	Observing that $|\Delta (\delta Y_{s,t})^2| = |\delta(\Delta Y)_{s,t}|~|\delta Y^1_{s,t}+\delta Y^2_{s,t}|$ and using the Cauchy-Schwarz inequality, one gets that for any subdivision $\pi=(t_i)$ of $[s,t]$,
	\begin{align}\label{eq:boundDeltadeltaY}
	\sum_{i} \left|\Delta (\delta Y_{t_i,t_{i+1}})^2\right|^{\frac{p}{2}}  &\leq \left( \sum_{i} |\delta (\Delta Y)_{t_i,t_{i+1}}|^p   \times \sum_{i} |\delta Y^1_{s,t}+\delta Y^2_{s,t}|^{p} \right)^{\frac{1}{2}} \nonumber\\
	&\leq \|\Delta Y\|_{\pv,[s,t]}^{\frac{p}{2}} \|Y^1+Y^2\|_{\pv,[s,t]}^{\frac{p}{2}} \nonumber\\
	&\leq C \|\Delta Y\|_{\pv,[s,t]}^{\frac{p}{2}} ~\kappab_{\RP,K}(s,t)^{\frac{1}{2}},
	\end{align}
	using \eqref{eq:boundYipv} in the last inequality.
	Now, using the smoothness of $\sigma$, observe that
	\begin{align*}
	|\delta(\Delta \sigma'\sigma(Y))_{s,t}| &= |\Delta Y_t \int_0^1 (\sigma'\sigma)'(Y^2_t + \Delta Y_t~ u) \dd u -  \Delta Y_s \int_0^1 (\sigma'\sigma)'(Y^2_s + \Delta Y_s~ u) \dd u | \\
	&\leq |\delta(\Delta Y)_{s,t} \int_0^1 (\sigma'\sigma)'(Y^2_t + \Delta Y_t u) \dd u| \\
	&\hspace{1cm}+ |\Delta Y_s|  \int_0^1 |(\sigma'\sigma)'(Y^2_t + \Delta Y_t~ u)- (\sigma'\sigma)'(Y^2_s + \Delta Y_s~ u)| \dd u  \\
	&\leq C\left( |\delta(\Delta Y)_{s,t}| + |\Delta Y_s| \left( |\delta Y^2_{s,t}| + |\delta(\Delta Y)_{s,t}| \right) \right) \\
	&\leq C\left( |\delta(\Delta Y)_{s,t}| + |\Delta Y_s| \left( |\delta Y^1_{s,t}| + |\delta Y^2_{s,t}| \right) \right).
	\end{align*}
	Therefore, using again \eqref{eq:boundYipv},
	\begin{align}\label{eq:deltaDeltasig}
	\|\Delta\sigma'\sigma(Y)\|_{\pv,[s,t]} \leq C \left(\|\Delta Y\|_{\pv,[s,t]} + \|\Delta Y\|_{\infty,[s,t]} ~\kappab_{\RP,K}(s,t)^{\frac{1}{p}} \right) .
	\end{align}
	
	\textbf{$4^{\text{th}}$ step.} The inequalities \eqref{eq:boundYipv}, \eqref{eq:boundRRspv}, \eqref{eq:boundDeltadeltaY} and \eqref{eq:deltaDeltasig} plugged into \eqref{eq:boundRDeltaY} now provide that for any $(s,t)\in\mathcal{S}_{[0,T]}$ such that $|t-s|\leq \delta_\RP$,
	\begin{align*}
	\|R^{\Delta\sigma(Y)} \|_{\ppv,[s,t]} &\leq C_p \|R^{\Delta \sigma(Y)}\|_{\ppv,[s,t]} \|X\|_{\pv,[s,t]} + C \bigg\{ w_K(s,t) + \|\Delta Y\|_{\infty,[s,t]} \left( \kappab_{\RP,K}(s,t)^{\frac{2}{p}} + \|\mathbb{X}\|_{\ppv,[s,t]}\right)\\
	&\quad\quad\quad + \|\Delta Y\|_{p,[s,t]} \kappab_{\RP,K}(s,t)^{\frac{1}{p}} + \|\Delta Y\|_{\infty,[s,t]} \kappab_{\RP,K}(s,t)^{\frac{2}{p}} \\
	&\quad\quad\quad +\left(\|\Delta Y\|_{\pv,[s,t]} + \|\Delta Y\|_{\infty,[s,t]} ~\kappab_{\RP,K}(s,t)^{\frac{1}{p}} \right) \|\mathbb{X}\|_{\ppv,[s,t]} \bigg\} .
	\end{align*}
	Besides, since $\kappab_{\RP,K}(s,t)^{\frac{1}{p}} \|\mathbb{X}\|_{\ppv,[s,t]} \leq \kappab_{\RP,K}(s,t)^{\frac{2}{p}}$, we get that for any $(s,t)\in\mathcal{S}_{[0,T]}$ such that $|t-s|\leq \delta_\RP$,
	\begin{align*}
	\|R^{\Delta \sigma(Y)} \|_{\ppv,[s,t]} \leq C_p \|R^{\Delta \sigma(Y)}\|_{\ppv,[s,t]} \|X\|_{\pv,[s,t]} + C \Big\{ w_K(s,t) + \|\Delta Y\|_{\infty,[s,t]} \kappab_{\RP,K}(s,t)^{\frac{2}{p}} \\
	\hspace{1cm} + \|\Delta Y\|_{\pv,[s,t]} \kappab_{\RP,K}(s,t)^{\frac{1}{p}}\Big\} .
	\end{align*}
	In view of the definition of $M$ and $\delta_\RP$ in the first Step, there is $\|X\|_{p,[s,t]} \leq (2C_p)^{-1}$ for any $|t-s|\leq \delta_\RP$. Thus for any $(s,t)\in\mathcal{S}_{[0,T]}$ such that $|t-s|\leq \delta_\RP$, one gets that
	\begin{align}\label{eq:boundRDeltasigY}
	\|R^{\Delta\sigma(Y)} \|_{\ppv,[s,t]} \leq 2C \left\{w_K(s,t) + \|\Delta Y\|_{\infty,[s,t]} \kappab_{\RP,K}(s,t)^{\frac{2}{p}}+ \|\Delta Y\|_{\pv,[s,t]}\kappab_{\RP,K}(s,t)^{\frac{1}{p}}\right\}.
	\end{align}
	
	\textbf{$5^{\text{th}}$ step.} It remains to bound $\|\Delta Y\|_{\pv,[s,t]}$. From the definition of $\Delta Y$ and Inequality \eqref{eq:boundRI},
	\begin{align*}
	|\delta(\Delta Y)_{s,t}| = |\delta(\Delta K)_{s,t} + \int_s^t \Delta\sigma(Y)_u \dd\RP_u|&\leq w_K(s,t) + |\Delta\sigma(Y)_s||\delta X_{s,t}| + |\Delta \sigma'\sigma(Y)_s| |\mathbb{X}_{s,t}| \\
	&\quad+ C_p\left(\|R^{\Delta\sigma(Y)}\|_{\ppv,[s,t]} \|X\|_{\pv,[s,t]} + \|\Delta\sigma'\sigma(Y)\|_{\pv,[s,t]} \|\mathbb{X}\|_{\ppv,[s,t]}\right) .
	\end{align*}
	The previous bounds \eqref{eq:deltaDeltasig} and \eqref{eq:boundRDeltasigY} can now be used as follows:  $\forall(s,t)\in\mathcal{S}_{[0,T]}$ such that $|t-s|\leq \delta_\RP$,
	\begin{align*}
	\|\Delta Y\|_{\pv,[s,t]} &\leq  w_K(s,t) + C\bigg\{ \|\Delta Y\|_{\infty,[s,t]} \left(\|X\|_{\pv,[s,t]} + \|\mathbb{X}\|_{\ppv,[s,t]} \right)\\
	&\quad+ \|X\|_{\pv,[s,t]} \left(w_K(s,t) + \|\Delta Y\|_{\infty,[s,t]} \kappab_{\RP,K}(s,t)^{\frac{2}{p}}+ \|\Delta Y\|_{\pv,[s,t]}\kappab_{\RP,K}(s,t)^{\frac{1}{p}} \right) \\
	&\quad + \|\mathbb{X}\|_{\ppv,[s,t]} \left( \|\Delta Y\|_{\pv,[s,t]} + \|\Delta Y\|_{\infty,[s,t]} ~\kappab_{\RP,K}(s,t)^{\frac{1}{p}}   \right)\bigg\} \\
	&\leq \tfrac{1}{2} C_Y \|\Delta Y\|_{\pv,[s,t]} \left(\|X\|_{\pv,[s,t]}\kappab_{\RP,K}(s,t)^{\frac{1}{p}} + \|\mathbb{X}\|_{\ppv,[s,t]} \right) + w_K(s,t)\left(1+ C\|X\|_{p,[s,t]} \right) \\
	&\quad + C \|\Delta Y\|_{\infty,[s,t]} \left( \kappab_{\RP,K}(s,t)^{\frac{1}{p}} +\kappab_{\RP,K}(s,t)^{\frac{3}{p}}\right) \\
	&\leq C_Y \|\Delta Y\|_{\pv,[s,t]} \kappab_{\RP,K}(s,t)^{\frac{2}{p}} + C (w_K(s,t)+\|\Delta Y\|_{\infty,[s,t]}) \left(1+ \kappab_{\RP,K}(s,t)^{\frac{1}{p}} +\kappab_{\RP,K}(s,t)^{\frac{3}{p}}\right),
	\end{align*}
	where $C_Y>0$ depends only on $p$ and $\sigma$. Since 
	\[
	(s,t) \mapsto C (w_K(s,t)+\|\Delta Y\|_{\infty,[s,t]})^p \left(1+ \kappab_{\RP,K}(s,t)^{\frac{1}{p}} +\kappab_{\RP,K}(s,t)^{\frac{3}{p}}\right)^p
	\]
	is super-additive, we obtain by a classical argument that \(\forall (s,t)\in \mathcal{S}_{[0,T]}\),
	\begin{align*}
	\|\Delta Y\|_{\pv,[s,t]} &\leq \left(2(2C_Y)^{\frac{p}{2}} \right)^{\frac{p-1}{p}} \left( 1\vee \kappab_{\RP,K}(s,t)^{\frac{p-1}{p}}\right) \times C (w_K(s,t)+\|\Delta Y\|_{\infty,[s,t]})  \sum_{\substack{j=0\\j\neq 2}}^3 \kappab_{\RP,K}(s,t)^{\frac{j}{p}} ,
	\end{align*}
	so that for any \((s,t)\in \mathcal{S}_{[0,T]}\),
	\begin{align*}
	\|\Delta Y\|_{\pv,[s,t]} \kappab_{\RP,K}(s,t)^{\frac{1}{p}} 
	&\leq C (w_K(s,t)+\|\Delta Y\|_{\infty,[s,t]}) \left( \widetilde{\kappab}_{\RP,K}(s,t)^{\frac{1}{p}} \vee \widetilde{\kappab}_{\RP,K}(s,t)\right).
	\end{align*}
	By plugging the above bound into the right-hand side of  \eqref{eq:boundRDeltasigY}, one obtains part \ref{item:ii} of the desired result. For further use, note also that the previous bound can be used in \eqref{eq:deltaDeltasig} to get that for any $(s,t)\in \mathcal{S}_{[0,T]}$,
	\begin{align}\label{eq:deltaDeltasig2}
	\|\Delta\sigma'\sigma(Y)\|_{\pv,[s,t]} \|\mathbb{X}\|_{\ppv,[s,t]} &\leq C \left( \|\Delta Y\|_{\pv,[s,t]} \kappab_{\RP,K}(s,t)^{\frac{1}{p}} + \|\Delta Y\|_{\infty,[s,t]} \kappab_{\RP,K}(s,t)^{\frac{1}{p}} \right)\nonumber\\ 
	&\leq C (w_K(s,t)+\|\Delta Y\|_{\infty,[s,t]}) \left( \widetilde{\kappab}_{\RP,K}(s,t)^{\frac{1}{p}} \vee \widetilde{\kappab}_{\RP,K}(s,t)\right).
	\end{align}
\end{proof}


\section*{Acknowledgements}
This work was partially supported by the ECOS-Sud Program Chili-France C15E05 and the Math-AmSud project SARC. 
S.T. is partially supported by the Project Fondecyt N. 1171335. 
The authors are grateful to Paul Gassiat for pointing out a simplification at the beginning of the proof of uniqueness.

\end{document}